\documentclass[sn-mathphys,Numbered]{sn-jnl}
\usepackage[utf8]{inputenc}
\usepackage{xpatch}
\usepackage{fancyhdr}
\usepackage{latexsym}
\usepackage{eufrak}
\usepackage[T1]{fontenc}
\usepackage{tikz}
\usepackage{xparse}
\usepackage{enumitem}
\usepackage{indentfirst}
\usepackage{graphicx}
\usepackage{multirow}
\usepackage{amsmath,amssymb,amsfonts}
\usepackage{amsthm}
\usepackage{mathrsfs}
\usepackage[title]{appendix}
\usepackage{xcolor}
\usepackage{textcomp}
\usepackage{manyfoot}
\usepackage{booktabs}
\usepackage{algorithm}
\usepackage{algorithmicx}
\usepackage{algpseudocode}
\usepackage{listings}
\usepackage{bm}

\DeclareMathOperator*{\R}{Re}
\DeclareMathOperator*{\I}{Im}

\theoremstyle{thmstyleone}
\newtheorem{definition}{Definition}[section]
\newtheorem{theorem}{Theorem}[section]
\newtheorem{lemma}{Lemma}[section]

\newtheorem{remark}{Remark}[section]

\newtheorem{proposition}{Proposition}[section]

\numberwithin{equation}{section}
\begin{document}
\title[Article Title]{On well-posedness of the space-time fractional nonlinear Schr\"odinger equation}
\author*[1]{\fnm{Mingxuan} \sur{He}}\email{MingxuanHe001@126.com}
\author[1]{\fnm{Na} \sur{Deng}}\email{202121511122@smail.xtu.edu.cn}
\author[1]{\fnm{Lu} \sur{Zhang}}\email{lzhang@xtu.edu.cn}
\affil[1]{\orgdiv{School of Mathematics and Computational Science}, \orgname{Xiangtan University}, \orgaddress{\city{Xiangtan}, \postcode{411105}, \state{Hunan}, \country{China}}}
\abstract{We study the Cauhcy problem for space-time fractional nonlinear Schr\"odinger equation with a general nonlinearity. We prove the local well-posedness of it in fractional Sobolev spaces based on the decay estimates and H\"older type estimates. Due to the lack of the semigroup structure of the solution operators, we deduce the decay estimates and H\"older type estimates via the asymptotic expansion of the Mittag-Leffler functions and Bessel functions. In particular, these results also show the dispersion of the solutions.}
\keywords{Space-time fractional Schr\"odinger equation, Well-posedness, Dispersion, Decay and H\"older type estimates, A priori estimates}
\maketitle
\section{Introduction}
\subsection{Notations}
By $a\lesssim b$, we mean that there exists a positive constant $C$ such that $a\leq Cb$ and by $a\sim b$, we mean that $a\lesssim b\lesssim a$. We denote the $\max\{a,b\}$ by the notation $a\vee b$ and $\min\{a,b\}$ by the notation $a\wedge b$. The Fourier transform of $u$ with respect to the space variable $x$ will be written as $\mathscr{F}u$ or $\widehat{u}$ and the inverse Fourier transform of $u$ with respect to the space variable $x$ will be written as $\mathscr{F}^{-1}u$ or $u^\vee$.

In this paper, we use $L_x^p$ to denote $L^p(\mathbb{R}^n)$ and $H_x^{s,p}$ to denote the fractional Sobolev spaces $H^{s,p}(\mathbb{R}^n)$ whose definition is $H_x^{s,p}=\left\langle\nabla\right\rangle^{-s}L_x^p$ and norm is $\left\lVert\left\langle\nabla\right\rangle^su\right\rVert_{L_x^p}$. In particular, we will abbreviate $H_x^{s,2}$ to $H_x^s$. $\left\langle\nabla\right\rangle^s$ denotes the Bessel potential of order $-s$ (i.e. $\left\langle\nabla\right\rangle^su=\mathscr{F}^{-1}\left(\left\langle\xi\right\rangle^s\widehat{u}\right)=\mathscr{F}^{-1}\left(\left(1+\left\lvert\xi\right\rvert^2\right)^{\frac{s}{2}}\widehat{u}\right)$), while $\left\lvert\nabla\right\rvert^s$ denotes the Riesz potential of order $-s$ (i.e. $\left\lvert\nabla\right\rvert^su=\mathscr{F}^{-1}\left(\left\lvert\xi\right\rvert^s\widehat{u}\right)$).

Let $X$ be a Banach space. We use $L_T^pX$ to denote $L^p\left((0,T);X\right)$ and $L_t^pX$ to denote $L^p\left((0,\infty);X\right)$ for $1\leq p\leq\infty$.

In the end, the constant $\sigma$ used throughout this paper stands for $\sigma=\frac{\alpha}{2\beta}$.
\subsection{Background and main results}
Over the past few centuries, fractional calculus has been discussed and studied by mathematicians such as Leibniz ,Euler, Abel, Riemann and Liouville as a pure mathematical technique with no practical applications. But since the last century, fractional calculus has proved to be useful in most areas of science due to the nonlocal characteristics of fractional differentiation, and has gradually been taken seriously by mathematicians, physicists, engineers and economists.

One of the most important physical models that has been extensively studied by many contributing authors such as Kato\cite{On-nonlinear-Schrodinger-equations, On-nonlinear-Schrodinger-equations-II-HS-solutions-and-unconditional-well-posedness}, Ginibre et al.\cite{On-a-class-of-nonlinear-Schrodinger-equations-I-The-Cauchy-problem-general-case, The-global-Cauchy-problem-for-the-nonlinear-Schrodinger-equation-revisited}, Nakamura et al.\cite{Modified-Strichartz-estimates-with-an-application-to-the-critical-nonlinear-Schrodinger-equation, Strichartz-type-estimates-in-mixed-Besov-spaces-with-application-to-critical-nonlinear-Schrodinger-equations}, Cazenave\cite{Semilinear-Schrodinger-equations}, Kenig\cite{Small-solutions-to-nonlinear-Schrodinger-equations} is the Schr\"odinger equation. The fractional generalization of Schr\"odinger equation is mainly divided into three fields. One of them is the spatial fractional version of the Schr\"odinger equation introduced by Laskin\cite{Fractional-Schrodinger-equation} as a fundamental equation of the fractional quantum mechanics\cite{Fractional-quantum-mechanics, Fractional-quantum-mechanics-and-Levy-path-integrals, Fractals-and-quantum-mechanics} whose form is given by $i\partial_tu-\left(-\Delta\right)^\beta u+F(u)=0$ where $\left(-\Delta\right)^\beta$ denotes the fractioal Laplacian which is the Fourier multiplier of symbol $\left\lvert\xi\right\rvert^{2\beta}$. For results on the well-posedness of it, we refer readers to \cite{On-Fractional-Schrodinger-Equations-in-sobolev-spaces,Existence-of-the-global-smooth-solution-to-the-period-boundary-value-problem-of-fractional-nonlinear-Schrodinger-equation,Global-Well-Posedness-for-the-Fractional-Nonlinear-Schrodinger-Equation}. The second is the time fractional Schr\"odinger equation which is given by $i^\Box\partial_t^\alpha u+\Delta u+F(u)=0$ where $\Box$ could be $1$ or $\alpha$. The case $\Box=\alpha$ is introduced by Naber\cite{Time-fractional-Schrodinger-equation} who use Wick rotation to raise a fractional power of $i$ which turns out to be the classical Schr\"odinger equations with a time dependent Hamiltonian. The case $\Box=1$ is introduced by Achar et al.\cite{Time-Fractional-Schrodinger-Equation-Revisited} who derive it using the Feynman paths method. $\partial_t^\alpha$ with order $\alpha\in(0,1)$ introduced by Caputo\cite{Linear-Models-of-Dissipation-whose-Q-is-almost-Frequency-Independent-II} denotes the Caputo derivative and is given by
\[ \partial_t^\alpha u=\frac{d}{dt}\left(\frac{1}{\Gamma(1-\alpha)}\int_0^t\left(t-\tau\right)^{-\alpha}\left(u(\tau)-u(0)\right)d\tau\right). \]

Caputo derivative is widely used in several scientific fields such as statistical mechanics, theoretical physics, theoritical chemistry, fluid mechanics and mathematical finance. We refer readers to \cite{Fractional-differential-equations-an-introduction-to-fractional-derivatives-fractional-differential-equations-to-methods-of-their-solution-and-some-of-their-applications, Theory-and-Applications-of-Fractional-Differential-Equations} for more details of the Caputo derivative. Some results about the case $\Box=\alpha$ can be found in the works\cite{The-Time-Fractional-Schrodinger-Equation-on-Hilbert-Space,The-asymptotic-behavior-of-the-time-fractional-Schrodinger-equation-on-Hilbert-space} and the case $\Box=1$ can be found in the works\cite{Fractional-Schrodinger-equations-with-potential-and-optimal-controls,The-well-posedness-for-fractional-nonlinear-Schrodinger-equations,Duhamels-formula-for-time-fractional-Schrodinger-equations,On-the-Stability-of-Time-Fractional-Schrodinger-Differential-Equations}. The third is the space-time fractional Schr\"odinger equation given by $i^\Box\partial_t^\alpha u-\left(-\Delta\right)^\beta u+F(u)=0$. Similarly, $\Box$ could be $1$ or $\alpha$. For the case $\Box=\alpha$, Lee\cite{Strichartz-estimates-for-space-time-fractional-Schrodinger-equations} derived the Strichartz estimates of the solution operator of it. Grande\cite{Space-Time-Fractional-Nonlinear-Schrodinger-Equation} studied the well-posedness and ill-posedness of it in one dimension with the nonlinearity in the form of even degree polynomials. Dong et al.\cite{Space-time-fractional-Schrodinger-equation-with-time-independent-potentials} obtained some results of it with time-independent potentials.

The case $\Box=1$ is of interest for this article. More precisely, this paper is concerned with the Cauchy problem for the space-time fractional nonlinear Schr\"odinger equation
\begin{equation}
\begin{cases}
i\partial_t^\alpha u-\left(-\Delta\right)^\beta u+F(u)=0,\quad&x\in\mathbb{R}^n,\;t>0,\\
u(0,x)=u_0(x),\quad&x\in\mathbb{R}^n.
\end{cases}
\label{9.5 space time fractional Schrodinger equation}
\end{equation}

In previous studies on $(\ref{9.5 space time fractional Schrodinger equation})$, more of them used numerical methods to study it; see \cite{A-numerical-method-for-solving-the-time-fractional-Schrodinger-equation, Application-of-Residual-Power-Series-Method-for-the-Solution-of-Time-fractional-Schrodinger-Equations-in-One-dimensional-Space, Jacobi-spectral-collocation-approximation-for-multi-dimensional-time-fractional-Schrodinger-equations} for instance and references therein. There are few studies on the well-posedness and behavior of solutions. Although Su et al.\cite{Dispersive-estimates-for-time-and-space-fractional-Schrodinger-equations, Local-well-posedness-of-semilinear-space-time-fractional-Schrodinger-equation} gave the dispersive estimates of the solution operator in the homogeneous case and studied the local well-posedness of it in the space $C_TL_x^r\cap L_T^qL_x^p$ with $F(u)=\mu\left\lvert u\right\rvert^\theta u$ under certain technical conditions for $(q,p,r,\theta)$, their results have certain limitations, the nonlinearity being so special and the space where the solution exists being so wide that lose the regularity.

In this article, we will examine this equation in a broader sense, that is, the nonlinearity we consider is more general and the spaces the solution exists in have arbitrary regularity and integrability. To be more precise, the nonlinearity $F(\cdot)$ is a complex function with the following hypotheses:
\begin{equation}
F\in C^1\left(\mathbb{C};\mathbb{C}\right),\quad F(0)=0,\quad\left\lvert F'(z)\right\rvert\lesssim\left\lvert z\right\rvert^{p-1}\;\mathrm{for}\;1\leq p<\infty,
\label{hypothesis on nonlinear F}
\end{equation}
\begin{equation}
\I\left(\overline{z}F(z)\right)=0,
\label{10.8 hypothesis on nonlinear F2}
\end{equation}
\begin{equation}
\begin{aligned}
&\mathrm{There\;exists\;a\;complex\;function}\;G(\cdot)\in C\left(\mathbb{C};\mathbb{R}\right)\;\mathrm{satisfying}\\
&G(z)\geq0\;\mathrm{for\;any}\;z\in\mathbb{C}\;\mathrm{such\;that\;for\;any\;complex\;function}\;u\\
&\mathrm{we\;have}\;\R\left(F(u)\partial_t^\alpha\overline{u}\right)\lesssim-\partial_t^\alpha G(u).
\end{aligned}
\label{10.8 hypothesis on nonlinear F3}
\end{equation}

It is well-known that the mass ($L_x^2$ norm) and the energy of the solution to $(\ref{9.5 space time fractional Schrodinger equation})$ in the case $\alpha=1$ is conserved with respect to time but in the case $\alpha\in(0,1)$ is not. Specifically, we have
\begin{equation}
\left\lVert u(t)\right\rVert_{L_x^2}\lesssim\left\lVert u_0\right\rVert_{L_x^2},\quad I(u(t))\lesssim I(u_0),
\label{11.11 nonconservation law}
\end{equation}
where the energy is defined by
\[ I(u(t))=\left\lVert\left(-\Delta\right)^{\frac{\beta}{2}}u(t)\right\rVert_{L_x^2}^2+\int_{\mathbb{R}^n}G(u(t))dx. \]

The proof will be given in Lemmas $\ref{10.8A priori estimates1}$ and $\ref{10.8A priori estimates2}$.

By Fourier transform and the theory of fractional ordinary differential equations, we can write the solution of $(\ref{9.5 space time fractional Schrodinger equation})$ as (see Appendix $\ref{20240521The derivation of integral equation}$ for details):
\begin{equation}
u=S_tu_0+i\mathcal{M}F(u),
\label{9.5 The equivalent integral equation equation1}
\end{equation}
where
\[ \mathcal{M}v(t)=\int_0^tP_{t-\tau}v(\tau)d\tau, \]
and
\begin{align*}
&S_t\phi=\mathscr{F}^{-1}\left(a_t(\xi)\widehat{\phi}\right),\quad a_t(\xi)=E_{\alpha,1}\left(-i\left\lvert\xi\right\rvert^{2\beta}t^\alpha\right),\\
&P_t\phi=\mathscr{F}^{-1}\left(b_t(\xi)\widehat{\phi}\right),\quad b_t(\xi)=t^{\alpha-1}E_{\alpha,\alpha}\left(-i\left\lvert\xi\right\rvert^{2\beta}t^\alpha\right).
\end{align*}

$E_{\alpha,1}$ and $E_{\alpha,\alpha}$ denote the Mittag-Leffler function whose definition is
\[ E_{\alpha,\beta}(z)=\sum\limits_{k=0}^\infty\frac{z^k}{\Gamma(\alpha k+\beta)},\quad \alpha,\beta, z\in\mathbb{C}. \]

As in the case of integer order $(\alpha=1)$, $S_t, P_t$ will be called evolution operators.

In previous studies, to estimate the evolution operators, people exploited the relationship between Mittag-Leffler function and Mainardi function to obtain
\begin{equation}
\label{20240503ML Mainardi function}
S_t=\int_0^\infty M_\alpha(\theta)e^{-i\theta t^\alpha\left(-\Delta\right)^\beta}d\theta,\quad P_t=\int_0^\infty\alpha\theta M_\alpha(\theta)t^{\alpha-1}e^{-i\theta t^\alpha\left(-\Delta\right)^\beta}d\theta.
\end{equation}
If we already have some estimates for the fractional Schr\"odinger operator $e^{-it\left(-\Delta\right)^\beta}$, we can address the same estimates by the property of Mainardi function (see $(\ref{20240503Mainardi function integral})$). We refer readers to Appendix $\ref{20240503The derivation}$ for detailed derivation. Peng et al.\cite{The-well-posedness-for-fractional-nonlinear-Schrodinger-equations} obtain the decay estimates in the case $\beta=1$ in such a way. However, in this way, we cannot obtain an estimate that $e^{-it\left(-\Delta\right)^\beta}$ does not have. Indeed, due to the lack of the dispersive estimates for $e^{-it\left(-\Delta\right)^\beta}$, we cannot obtain the dispersive estimates as well as the decay estimates of $S_t$ and $P_t$ using the method above. Besides, the H\"older type estimates will also be complicated. Motivated by Grande\cite{Space-Time-Fractional-Nonlinear-Schrodinger-Equation}, we take advantage of the asymptotic expansion of Mittag-Leffler function to decompose the evolution operators into low frequency terms plus high frequency terms and estimates them respectively, so that we can obtain the $L^q-L^r$ decay estimates and the $L^q-L^r$ H\"older type estimates. A more detailed discussion will be given in $\S\ref{Estimates of the linear operators}$.

The first part of the main results is the local well-posedness of $(\ref{9.5 space time fractional Schrodinger equation})$. We first give the meaning of the continuous dependence. Let $u\in Y$ be the unique solution of $(\ref{9.5 space time fractional Schrodinger equation})$ with initial data $u_0\in X$ and $u_0^k\to u_0$ in $X$ as $k\to\infty$. Let $u_k$ be the solution of $(\ref{9.5 space time fractional Schrodinger equation})$ with initial data $u_k(0)=u_0^k$. We call the map $u_0\mapsto u(t)$ is continuous if the solution $u_k$ exists in $Y$ when $k$ is sufficiently large, and $u_k\to u$ in $Y$ as $k\to\infty$.
\begin{theorem}
\label{20240515locla well-posendess}
Let $s\geq0$, $\beta>\frac{n}{2}$, $2\leq p<\infty$, $1\leq r\leq\infty$, $1\leq q\leq\infty$ satisfying $\frac{\sigma n}{q}<\frac{\alpha}{p-1}$ and $F$ satisfy the hypothesis $(\ref{hypothesis on nonlinear F})$. For any $\gamma\in\left(\frac{\sigma n}{q},\frac{\alpha}{p-1}\right)$, there exists $T_{\max}>0$ satisfying
\[ T_{\max}<\infty\Longrightarrow\lim\limits_{t\uparrow T_{\max}}\left\lVert u(t)\right\rVert_{H_x^{s,r}}=\infty,\;{\rm or}\;\lim\limits_{t\uparrow T_{\max}}t^\gamma\left\lVert u(t)\right\rVert_{L_x^\infty}=\infty \]
such that $(\ref{9.5 space time fractional Schrodinger equation})$ admits a unique solution in the class
\[ u\in C\left(\left[0,T_{\max}\right);H_x^{s,r}\right),\quad t^\gamma u\in C\left(\left[0,T_{\max}\right);L_x^\infty\right) \]
with $u_0\in L_x^q\cap H_x^{s,r}$. If, in addition, $\gamma\in\left(\frac{\sigma n}{q},\frac{\alpha}{p}\right)$, the mapping $u_0\mapsto u(t)$ is continuous.
\end{theorem}
\begin{theorem}
\label{20240518locla well-posendessII}
Let $0\leq s<\frac{2\beta}{\alpha}$, $\beta>\frac{n}{2}$, $2\leq p<\infty$, $1\leq r\leq\infty$ satisfying $\frac{\sigma n}{r}<\left(\frac{\alpha}{p-1}\wedge\frac{1-\sigma s}{p-1}\wedge\frac{1}{p}\right)$ and $F$ satisfy the hypothesis $(\ref{hypothesis on nonlinear F})$. For any $\gamma\in\left(\frac{\sigma n}{r},\frac{\alpha}{p-1}\wedge\frac{1-\sigma s}{p-1}\wedge\frac{1}{p}\right)$, there exists
$T_{\max}>0$ satisfying
\[ T_{\max}<\infty\Longrightarrow\lim\limits_{t\uparrow T_{\max}}t^{\sigma s}\left\lVert u(t)\right\rVert_{H_x^{s,r}}=\infty,\;{\rm or}\;\lim\limits_{t\uparrow T_{\max}}t^\gamma\left\lVert u(t)\right\rVert_{L_x^\infty}=\infty \]
such that $(\ref{9.5 space time fractional Schrodinger equation})$ admits a unique solution in the class
\[ t^{\sigma s}u\in C\left(\left[0,T_{\max}\right);H_x^{s,r}\right),\quad t^\gamma u\in C\left(\left[0,T_{\max}\right);L_x^\infty\right) \]
with $u_0\in L_x^r$. If, in addition, $\gamma\in\left(\frac{\sigma n}{r},\frac{\alpha}{p}\wedge\frac{1-\sigma s}{p-1}\right)$, the mapping $u_0\mapsto u(t)$ is continuous.
\end{theorem}
\begin{remark}
We cannot get the blow-up criterion by iterative method like what we do in the integer case ($\alpha=1$) for the time fractional equation since the time fractional equation is not invariant under time translation. Motivated by the works\cite{Well-posedness-and-blow-up-results-for-a-class-of-nonlinear-fractional-Rayleigh-Stokes-problem,Well-posedness-results-for-a-class-of-semilinear-time-fractional-diffusion-equations,On-well-posedness-of-semilinear-Rayleigh-Stokes-problem-with-fractional-derivative-on}, we prove the continuation of the solution by contraction mapping theorem and then prove the blow-up criterion by continuation.
\end{remark}
The second part of the main results is the global well-posedness of $(\ref{9.5 space time fractional Schrodinger equation})$ with small initial data.
\begin{theorem}
\label{20240517global solution with small data II}
Let $s\geq0$, $\beta>\frac{n}{2}$, $\frac{1}{1-\alpha}\vee2<p<\infty$, $1\leq r\leq\infty$ and $F$ satisfy the hypothesis $(\ref{hypothesis on nonlinear F})$. There exists $\gamma>0$ such that $(\ref{9.5 space time fractional Schrodinger equation})$ admits a unique solution in the class
\[ u\in C\left([0,\infty);H_x^{s,r}\right),\quad t^\gamma u\in C\left([0,\infty);L_x^\infty\right) \]
with a small initial data $u_0\in L_x^{\frac{n(p-1)}{2\beta}}\cap H_x^{s,r}$ in the sense that
\[ \left\lVert u_0\right\rVert_{L_x^{\frac{n(p-1)}{2\beta}}}+\left\lVert u_0\right\rVert_{H_x^{s,r}}<\delta, \]
where $\delta$ is sufficiently small.
\end{theorem}
\begin{theorem}
\label{20240517global solution with small data III}
Let $0\leq s<\frac{2\beta}{\alpha}-2\beta$, $\beta>\frac{n}{2}$, $\frac{1}{1-\alpha}\vee2<p<\infty$, $1\leq r\leq\infty$ and $F$ satisfy the hypothesis $(\ref{hypothesis on nonlinear F})$. There exists $\gamma>0$ such that $(\ref{9.5 space time fractional Schrodinger equation})$ admits a unique solution in the class
\[ u\in C\left([0,\infty);L_x^r\right),\quad t^\gamma u\in C\left([0,\infty);L_x^\infty\right),\quad t^{\sigma s}u\in C\left([0,\infty);\dot{H}_x^{s,r}\right) \]
with a small initial data $u_0\in L_x^{\frac{n(p-1)}{2\beta}}\cap L_x^r$ in the sense that
\[ \left\lVert u_0\right\rVert_{L_x^{\frac{n(p-1)}{2\beta}}}+\left\lVert u_0\right\rVert_{L_x^r}<\delta, \]
where $\delta$ is sufficiently small.
\end{theorem}
The third part of the main results is the global well-posedness of $(\ref{9.5 space time fractional Schrodinger equation})$ with arbitrary initial data.
\begin{theorem}
\label{20240517global solution arbitrary data}
Let $s\geq0$, $\beta>\frac{n}{2}$, $1\leq r\leq\infty$, $1\leq q\leq\infty$ satisfying $\frac{\sigma n}{q}<\frac{\alpha}{p-1}$, $F$ satisfy the hypotheses $(\ref{hypothesis on nonlinear F})$, $(\ref{10.8 hypothesis on nonlinear F2})$ and $(\ref{10.8 hypothesis on nonlinear F3})$ and
\[ \begin{cases}
2\leq p\leq\frac{2n}{n-\beta},\quad&\beta<n,\\
2\leq p<\infty,\quad&\beta\geq n.
\end{cases} \]

For any $\gamma\in\left(\frac{\sigma n}{q},\frac{\alpha}{p-1}\right)$, $(\ref{9.5 space time fractional Schrodinger equation})$ admits a unique solution in the class
\[ u\in C\left(\left[0,\infty\right);H_x^{s,r}\right),\quad t^\gamma u\in C\left(\left[0,\infty\right);L_x^\infty\right) \]
with $u_0\in H_x^{\frac{\beta}{2}}\cap L_x^q\cap H_x^{s,r}$ satisfying $\int_{\mathbb{R}^n}G(u_0)dx<\infty$.
\end{theorem}
\begin{theorem}
\label{20240519global solution arbitrary dataII}
Let $0\leq s<2\beta$, $\beta>\frac{n}{2}$, $2\leq p<\infty$, $1\leq r\leq\infty$ satisfying $\frac{\sigma n}{r}<\left(\frac{\alpha-\sigma s}{p-1}\wedge\frac{1}{p}\right)$, $F$ satisfy the hypotheses $(\ref{hypothesis on nonlinear F})$, $(\ref{10.8 hypothesis on nonlinear F2})$ and $(\ref{10.8 hypothesis on nonlinear F3})$ and
\[ \begin{cases}
2\leq p\leq\frac{2n}{n-\beta},\quad&\beta<n,\\
2\leq p<\infty,\quad&\beta\geq n.
\end{cases} \]

For any $\gamma\in\left(\frac{\sigma n}{r},\frac{\alpha-\sigma s}{p-1}\wedge\frac{1}{p}\right)$, $(\ref{9.5 space time fractional Schrodinger equation})$ admits a unique solution in the class
\[ t^{\sigma s}u\in C\left([0,\infty);H_x^{s,r}\right),\quad t^\gamma u\in C\left([0,\infty);L_x^\infty\right) \]
with $u_0\in L_x^r\cap H_x^{\frac{\beta}{2}}$ satisfying $\int_{\mathbb{R}^n}G(u_0)dx<\infty$.
\end{theorem}
\section{Some elementary estimates}
\label{11.11Preliminaries of the Besov space and some elementary estimates}
\subsection{Estimates of the evolution operators}
\label{Estimates of the linear operators}
Using the asymptotic expansion of the Mittag-Leffler function, we can decompose the evolution operators into
\begin{equation}
\label{9.23MLinequalities hold}
\begin{aligned}
&S_t\phi=S_t\chi_t(D)\phi-\frac{i}{\Gamma(1-\alpha)}t^{-\alpha}\left\lvert\nabla\right\rvert^{-2\beta}\chi_t^c(D)\phi+t^{-2\alpha}O\left(\left\lvert\nabla\right\rvert^{-4\beta}\right)\chi_t^c(D)\phi,\\
&P_t\phi=P_t\chi_t(D)\phi+\frac{1}{\Gamma(-\alpha)}t^{-\alpha-1}\left\lvert\nabla\right\rvert^{-4\beta}\chi_t^c(D)\phi+t^{-2\alpha-1}O\left(\left\lvert\nabla\right\rvert^{-6\beta}\right)\chi_t^c(D)\phi.
\end{aligned}
\end{equation}

The detailed derivation and the definition of the operators will be left to Appendix $\ref{20240502The derivation}$.

Using Lemmas $\ref{20240505Decay estimates St first}$-$\ref{20240505Holder estimates St third}$, we obtain the following two propositions (Propositions $\ref{9.24 decay estimates2}$ and $\ref{10.10 Holder type estimates}$).
\begin{proposition}[Decay estimates]
\label{9.24 decay estimates2}
Let $0\leq\theta<2\beta-n$ for $\beta>\frac{n}{2}$ and $1\leq q\leq r\leq\infty$. Then
\begin{align*}
&\left\lVert\left\lvert\nabla\right\rvert^\theta S_t\phi\right\rVert_{L_x^r}\lesssim t^{-\sigma\theta-\sigma n\left(\frac{1}{q}-\frac{1}{r}\right)}\left\lVert\phi\right\rVert_{L_x^q},\\
&\left\lVert\left\lvert\nabla\right\rvert^\theta P_t\phi\right\rVert_{L_x^r}\lesssim t^{-\sigma\theta-\sigma n\left(\frac{1}{q}-\frac{1}{r}\right)+\alpha-1}\left\lVert\phi\right\rVert_{L_x^q}.
\end{align*}
\end{proposition}
\begin{proposition}[H\"older type estimates]
\label{10.10 Holder type estimates}
Let $0\leq\theta<2\beta-n$ for $\beta>\frac{n}{2}$ and $1\leq q\leq r\leq\infty$. For any $t_1,t_2>0$, we have
\begin{align*}
&\left\lVert\left\lvert\nabla\right\rvert^\theta\left(S_{t_1}-S_{t_2}\right)\phi\right\rVert_{L_x^r}\lesssim\left(t_1\wedge t_2\right)^{-1}\left\lvert t_1^{1-\sigma\theta-\sigma n\left(\frac{1}{q}-\frac{1}{r}\right)}-t_2^{1-\sigma\theta-\sigma n\left(\frac{1}{q}-\frac{1}{r}\right)}\right\rvert\left\lVert\phi\right\rVert_{L_x^q},\\
&\left\lVert\left\lvert\nabla\right\rvert^\theta\left(P_{t_1}-P_{t_2}\right)\phi\right\rVert_{L_x^r}\lesssim\left\lvert t_1^{\alpha-1-\sigma\theta-\sigma n\left(\frac{1}{q}-\frac{1}{r}\right)}-t_2^{\alpha-1-\sigma\theta-\sigma n\left(\frac{1}{q}-\frac{1}{r}\right)}\right\rvert\left\lVert\phi\right\rVert_{L_x^q}.
\end{align*}
\end{proposition}
\begin{remark}
One can easily see that the above propositions (Propositions $\ref{9.24 decay estimates2}$ and $\ref{10.10 Holder type estimates}$) is valid if we replace the norms of $L_x^r$ and $L_x^q$ by $H_x^{s,r}$ and $H_x^{s,q}$ or $B_{r,l}^s$ and $B_{q,l}^s$ respectively where $s\in\mathbb{R}$ and $1\leq l\leq\infty$.
\end{remark}
\subsection{Estimates of the nonlinearity}
\label{Estimates of the nonlinearity}
Using H\"older's inequality, we can easily obtain the following lemma.
\begin{lemma}
\label{10.10 Main results theorem1 proof lemma2}
Let $1\leq r\leq\infty$, $1\leq p<\infty$ and $F\in C^1\left(\mathbb{C};\mathbb{C}\right)$ satisfy $\left\lvert F'(\xi)\right\rvert\lesssim\left\lvert\xi\right\rvert^{p-1}$. Then the following estimate holds:
\[ \left\lVert F(u)-F(v)\right\rVert_{L_x^r}\lesssim\left(\left\lVert u\right\rVert_{L_x^\infty}^{p-1}+\left\lVert v\right\rVert_{L_x^\infty}^{p-1}\right)\left\lVert u-v\right\rVert_{L_x^r}. \]
\end{lemma}
\begin{lemma}
\label{10.9Estimates of the nonlinearity Lemma1}
Let $s\geq0$, $1\leq r\leq\infty$, $1\leq p<\infty$ and $F\in C\left(\mathbb{C};\mathbb{C}\right)$ satisfying $\left\lvert F(\xi)\right\rvert\lesssim\left\lvert\xi\right\rvert^p$. $F$ maps $H_x^s\cap L_x^\infty$ boundedly with the estimate
\[ \left\lVert F(u)\right\rVert_{H_x^{s,r}}\lesssim\left\lVert u\right\rVert_{L_x^\infty}^{p-1}\left\lVert u\right\rVert_{H_x^{s,r}}. \]
\end{lemma}
\begin{proof}[Proof]
The case $s=0$ can be proved by H\"older's inequality. For the case $s>0$, using the identity{\cite[(6.1.2)]{Modern-Fourier-Analysis}}:
\[ \left\lvert\nabla\right\rvert^su=\pi^{-s-\frac{n}{2}}\frac{\Gamma\left(\frac{n+s}{2}\right)}{\Gamma\left(-\frac{s}{2}\right)}\int_{\mathbb{R}^n}\left\lvert x-y\right\rvert^{-n-s}u(y)dy, \]
we have
\begin{align*}
\left\lVert\left\lvert\nabla\right\rvert^sF(u)\right\rVert_{L_x^r}&=\left\lVert\pi^{-s-\frac{n}{2}}\frac{\Gamma\left(\frac{n+s}{2}\right)}{\Gamma\left(-\frac{s}{2}\right)}\int_{\mathbb{R}^n}\left\lvert x-y\right\rvert^{-n-s}F(u(y))dy\right\rVert_{L_x^r}\\
&=\left\lVert\pi^{-s-\frac{n}{2}}\frac{\Gamma\left(\frac{n+s}{2}\right)}{\Gamma\left(-\frac{s}{2}\right)}\int_{\mathbb{R}^n}\left\lvert x-y\right\rvert^{-n-s}\frac{F(u(y))}{\left\lvert u(y)\right\rvert^p}\left\lvert u(y)\right\rvert^pdy\right\rVert_{L_x^r}\\
&\leq\left\lVert\frac{F(u)}{\left\lvert u\right\rvert^p}\right\rVert_{L_x^\infty}\left\lVert\pi^{-s-\frac{n}{2}}\frac{\Gamma\left(\frac{n+s}{2}\right)}{\Gamma\left(-\frac{s}{2}\right)}\int_{\mathbb{R}^n}\left\lvert x-y\right\rvert^{-n-s}\left\lvert u(y)\right\rvert^pdy\right\rVert_{L_x^r}\\
&\lesssim\left\lVert u\right\rVert_{L_x^\infty}^{p-1}\left\lVert\pi^{-s-\frac{n}{2}}\frac{\Gamma\left(\frac{n+s}{2}\right)}{\Gamma\left(-\frac{s}{2}\right)}\int_{\mathbb{R}^n}\left\lvert x-y\right\rvert^{-n-s}\left\lvert u(y)\right\rvert dy\right\rVert_{L_x^r}\\
&=\left\lVert u\right\rVert_{L_x^\infty}^{p-1}\left\lVert\left\lvert\nabla\right\rvert^su\right\rVert_{L_x^r}.
\end{align*}

Therefore,
\[ \left\lVert F(u)\right\rVert_{H_x^{s,r}}\sim\left\lVert F(u)\right\rVert_{L_x^r}+\left\lVert\left\lvert\nabla\right\rvert^sF(u)\right\rVert_{L_x^r}\lesssim\left\lVert u\right\rVert_{L_x^\infty}^{p-1}\left\lVert u\right\rVert_{H_x^{s,r}}. \]
\end{proof}
\begin{lemma}
\label{10.249.24 Main results theorem1 proof lemma3}
Let $s\geq0$, $2\leq p<\infty$ and $F\in C^1\left(\mathbb{C};\mathbb{C}\right)$ satisfy $\left\lvert F'(\xi)\right\rvert\lesssim\left\lvert\xi\right\rvert^{p-1}$. Then
\begin{align*}
\left\lVert F(u)-F(v)\right\rVert_{H_x^{s,r}}\lesssim&\left(\left\lVert u\right\rVert_{L_x^\infty}^{p-2}+\left\lVert v\right\rVert_{L_x^\infty}^{p-2}\right)\left(\left\lVert u\right\rVert_{H_x^{s,r}}+\left\lVert v\right\rVert_{H_x^{s,r}}\right)\left\lVert u-v\right\rVert_{L_x^\infty}\\
&+\left(\left\lVert u\right\rVert_{L_x^\infty}^{p-1}+\left\lVert v\right\rVert_{L_x^\infty}^{p-1}\right)\left\lVert u-v\right\rVert_{H_x^{s,r}}.
\end{align*}
\end{lemma}
\begin{proof}[Proof]
By fractional Leibniz's rule{\cite[Theorem 1.4]{Exact-smoothing-properties-of-Schrodinger-semigroups}}, we have
\begin{align*}
&\left\lVert F(u)-F(v)\right\rVert_{H_x^{s,r}}\\
&\leq\int_0^1\left\lVert F'(v+t(u-v))(u-v)\right\rVert_{H_x^{s,r}}dt\\
&\lesssim\int_0^1\left\lVert F'(v+t(u-v))\right\rVert_{L_x^\infty}\left\lVert u-v\right\rVert_{H_x^{s,r}}+\left\lVert F'(v+t(u-v))\right\rVert_{H_x^{s,r}}\left\lVert u-v\right\rVert_{L_x^\infty}dt.
\end{align*}

Using Lemma $\ref{10.9Estimates of the nonlinearity Lemma1}$ we can obtain the result.
\end{proof}
\subsection{A priori estimates}
\label{11.5A priori estimates}
The following lemma tells us the mass of the solution to $(\ref{9.5 space time fractional Schrodinger equation})$ is not conserved with respect to time.
\begin{lemma}
\label{10.8A priori estimates1}
Let $u_0\in L_x^2$ and $F$ satisfy the hypothesis $(\ref{10.8 hypothesis on nonlinear F2})$. Then the solution of $(\ref{9.5 space time fractional Schrodinger equation})$ satisfies
\[ \left\lVert u(t)\right\rVert_{L_x^2}\leq\left\lVert u_0\right\rVert_{L_x^2}. \]
\end{lemma}
\begin{proof}[Proof]
The result can be proved by multiplying $(\ref{9.5 space time fractional Schrodinger equation})$ by $\overline{u}$ and considering the imaginary part and using the fact
\[ \R\left(\overline{u}\partial_t^\alpha u\right)\gtrsim\partial_t^\alpha\left\lvert u\right\rvert^2. \]
\end{proof}
We know from the following lemma that the energy of the solution to $(\ref{9.5 space time fractional Schrodinger equation})$ is not conserved with respect to time.
\begin{lemma}
\label{10.8A priori estimates2}
Let $u_0\in H_x^{\frac{\beta}{2}}$ such that $\int_{\mathbb{R}^n}G(u_0)dx<\infty$ and $F$ satisfy the hypothesis $(\ref{10.8 hypothesis on nonlinear F3})$. Then the solution of $(\ref{9.5 space time fractional Schrodinger equation})$ satisfies
\[ \left\lVert\left(-\Delta\right)^{\frac{\beta}{2}}u(t)\right\rVert_{L_x^2}^2+\int_{\mathbb{R}^n}G(u(t))dx\lesssim I(u_0), \]
where $I(u_0)=\left\lVert\left(-\Delta\right)^{\frac{\beta}{2}}u_0\right\rVert_{L_x^2}^2+\int_{\mathbb{R}^n}G(u_0)dx$.
\end{lemma}
\begin{proof}[Proof]
Multiplying $(\ref{9.5 space time fractional Schrodinger equation})$ by $\partial_t^\alpha\overline{u}$, considering the real part and using the fact
\[ \int_{\mathbb{R}^n}\left(-\Delta\right)^\beta u\partial_t^\alpha\overline{u}dx=\int_{\mathbb{R}^n}\left(-\Delta\right)^{\frac{\beta}{2}}u\partial_t^\alpha\left(-\Delta\right)^{\frac{\beta}{2}}\overline{u}dx \]
we can complete the proof.
\end{proof}
\begin{remark}
In view of Lemma $\ref{10.8A priori estimates1}$ and Lemma $\ref{10.8A priori estimates2}$, if $u_0\in H_x^{\frac{\beta}{2}}$ such that $\int_{\mathbb{R}^n}G(u_0)dx<\infty$ and $F$ satisfies hypotheses $(\ref{10.8 hypothesis on nonlinear F2})$ and $(\ref{10.8 hypothesis on nonlinear F3})$, the solution of $(\ref{9.5 space time fractional Schrodinger equation})$ satisfies
\begin{equation}
\left\lVert u(t)\right\rVert_{H_x^{\frac{\beta}{2}}}^2+\int_{\mathbb{R}^n}G(u(t))dx\lesssim E(u_0),
\label{10.8A priori estimates equation1}
\end{equation}
where
\[ E(u_0)=\left\lVert u_0\right\rVert_{H_x^{\frac{\beta}{2}}}^2+\int_{\mathbb{R}^n}G(u_0)dx. \]
\label{10.8A priori estimates3}
\end{remark}
\section{Proof of Theorems $\ref{20240515locla well-posendess}$ and $\ref{20240518locla well-posendessII}$}
Firstly, we define two function spaces.
\begin{align*}
&\mathcal{X}_\gamma^{s,r}(T):=\left\{u\in L_T^\infty H_x^{s,r}:t^\gamma u\in L_T^\infty L_x^\infty\right\},\\
&\overline{\mathcal{X}_\gamma^{s,r}(T)}:=\left\{u\in C_TH_x^{s,r}:t^\gamma u\in C_TL_x^\infty,\;\lim\limits_{t\to0}t^\gamma\left\lVert u(t)\right\rVert_{L_x^\infty}=0\right\}.
\end{align*}

Let $\Phi u=S_tu_0+i\mathcal{M}F(u)$.
\begin{lemma}
\label{20240516local wellposedness lemma}
$\Phi$ maps $\mathcal{X}_\gamma^{s,r}(T)$ into $\overline{\mathcal{X}_\gamma^{s,r}(T)}$ under the hypotheses in Theorem $\ref{20240515locla well-posendess}$.
\end{lemma}
\begin{proof}[Proof]
The proof is not difficult and we omit it.
\end{proof}
\begin{lemma}[{\cite[Lemma 7.1.2]{Geometric-Theory-of-Semilinear-Parabolic-Equations}}]
\label{10.27Gronwall inequality lemma}
Suppose $\beta>0$, $\gamma>0$, $\beta+\gamma>1$ and $a\geq0$, $b\geq0$, $u$ is nonnegative and $t^{\gamma-1}u(t)$ is locally integrable on $0\leq t<T$, and
\[ u(t)\leq a+b\int_0^t\left(t-\tau\right)^{\beta-1}\tau^{\gamma-1}u(\tau)d\tau \]
a.e. in $[0,T)$; then
\[ u(t)\leq a\mathbb{E}_{\beta,\gamma}\left(\left(b\Gamma(\beta)\right)^{\frac{1}{\nu}}t\right) \]
where $\nu=\beta+\gamma-1>0$, $\mathbb{E}_{\beta,\gamma}(t)=\sum\limits_{m=0}^\infty c_mt^{m\nu}$ with $c_0=1$, $\frac{c_{m+1}}{c_m}=\frac{\Gamma(m\nu+\gamma)}{\Gamma(m\nu+\gamma+\beta)}$ for $m\geq0$.
\end{lemma}
\begin{remark}
It's worth noting that
\[ \mathbb{E}_{\beta,\gamma}(t)\lesssim E_{\beta,\delta}\left(\left(\frac{\beta}{\nu}\right)^\beta t^\nu\right) \]
where $\delta=\frac{\beta\gamma+\nu}{2\nu}$.
\end{remark}

Before proving Theorem $\ref{20240515locla well-posendess}$, we shall prove a continuation result which will be helpful in proving the blow-up criteria. More precisely, we have the following lemma.
\begin{lemma}
\label{20240516continuation lemma}
Let $u\in\overline{\mathcal{X}_\gamma^{s,r}(T)}$ be the solution of $(\ref{9.5 space time fractional Schrodinger equation})$ provided in Theorem $\ref{20240515locla well-posendess}$ on $[0,T]$. Then $u$ can be extended to the interval $\left[0,\overline{T}\right]$ for some $\overline{T}>T$ uniquely such that the extended function is in the class $\overline{\mathcal{X}_\gamma^{s,r}\left(\overline{T}\right)}$ and also the solution of $(\ref{9.5 space time fractional Schrodinger equation})$ on $\left[0,\overline{T}\right]$.
\end{lemma}
\begin{proof}[Proof]
Define a function space as
\[ \Omega:=\left\{v\in\overline{\mathcal{X}_\gamma^{s,r}\left(\overline{T}\right)}:\begin{gathered}
v\equiv u\;on\;[0,T]\\
\sup\limits_{t\in\left[T,\overline{T}\right]}\left\lVert v(t)-u(T)\right\rVert_{H_x^{s,r}}+\sup\limits_{t\in\left[T,\overline{T}\right]}\left\lVert t^\gamma v(t)-T^\gamma u(T)\right\rVert_{L_x^\infty}\leq R
\end{gathered}\right\}, \]
where $R$ satisfies $\left\lVert u\right\rVert_{L_T^\infty H_x^{s,r}}+\left\lVert t^\gamma u\right\rVert_{L_T^\infty L_x^\infty}\leq R$. $\Omega$ is a complete metric space equipped with the metric
\[ d_{\overline{T}}(v,w):=\left\lVert v-w\right\rVert_{L_{\overline{T}}^\infty L_x^r}+\left\lVert t^\gamma(v-w)\right\rVert_{L_T^\infty L_x^\infty}. \]

Let $\Phi v=S_tu_0+i\mathcal{M}F(v)$. For any $v\in\Omega$, $\Phi v\in\overline{\mathcal{X}_\gamma^{s,r}\left(\overline{T}\right)}$ by Lemma $\ref{20240516local wellposedness lemma}$ and $\Phi v\equiv\Phi u\equiv u$ on $[0,T]$ since $u$ is a solution to $(\ref{9.5 space time fractional Schrodinger equation})$ on $[0,T]$. On $\left[T,\overline{T}\right]$, we have, by Propositions $\ref{9.24 decay estimates2}$ and $\ref{10.10 Holder type estimates}$ and Lemmas $\ref{10.10 Main results theorem1 proof lemma2}$ and $\ref{10.9Estimates of the nonlinearity Lemma1}$, that
\begin{align*}
\sup\limits_{t\in\left[T,\overline{T}\right]}\left\lVert\Phi v(t)-u(T)\right\rVert_{H_x^{s,r}}&\lesssim T^{-1}\left(\overline{T}-T\right)R+\int_T^t\left(t-\tau\right)^{\alpha-1}\tau^{-\gamma(p-1)}d\tau R^p\\
&\lesssim T^{-1}\left(\overline{T}-T\right)R+T^{-\gamma(p-1)}\left(t-T\right)^\alpha R^p,
\end{align*}
and
\begin{align*}
\sup\limits_{t\in\left[T,\overline{T}\right]}\left\lVert t^\gamma\Phi v(t)-T^\gamma u(T)\right\rVert_{L_x^\infty}&\lesssim\left(T^{-\frac{\sigma n}{q}}\left(\overline{T}-T\right)^\gamma+T^{\gamma-1}\left(\overline{T}-T\right)^{1-\frac{\sigma n}{q}}\right)R\\
&+\left(T^{-\gamma}\overline{T}^{\alpha-\frac{\sigma n}{r}-\gamma(p-2)}\left(\overline{T}-T\right)^\gamma+T^{\gamma(1-p)}\left(\overline{T}-T\right)^\alpha\right)R^p,
\end{align*}
which then follows that
\[ \sup\limits_{t\in\left[T,\overline{T}\right]}\left\lVert\Phi v(t)-u(T)\right\rVert_{H_x^{s,r}}+\sup\limits_{t\in\left[T,\overline{T}\right]}\left\lVert t^\gamma\Phi v(t)-T^\gamma u(T)\right\rVert_{L_x^\infty}\leq R \]
by choosing $\overline{T}$ be close enough to $T$ and hence $\Phi$ maps $\Omega$ into $\Omega$. For any $v,w\in\Omega$, by Propositions $\ref{9.24 decay estimates2}$ and $\ref{10.10 Holder type estimates}$ and Lemma $\ref{10.10 Main results theorem1 proof lemma2}$ we conclude that
\[ d_{\overline{T}}\left(\Phi v,\Phi w\right)\lesssim\left(T^{\gamma(1-p)}+\overline{T}^\gamma T^{-\gamma p}\right)\left(\overline{T}-T\right)^\alpha R^{p-1}d_{\overline{T}}(v,w). \]

Hence also we can choose $\overline{T}$ be close enough to $T$ such that $\Phi$ is a contraction on $\Omega$. An application of contraction mapping theorem leads to the result.
\end{proof}
\begin{proof}[\textbf{Proof of Theorem} $\mathbf{\ref{20240515locla well-posendess}}$]
\textbf{Step 1:}Define a complete metric space $E_\gamma^{s,r}(T)$ by
\[ E_\gamma^{s,r}(T):=\left\{u\in\mathcal{X}_\gamma^{s,r}(T):\left\lVert u\right\rVert_{L_T^\infty H_x^{s,r}}+\left\lVert t^\gamma u\right\rVert_{L_T^\infty L_x^\infty}\leq\left\lVert u_0\right\rVert_{L_x^q\cap H_x^{s,r}}\right\} \]
with its metric
\[ d_T(u,v):=\left\lVert u-v\right\rVert_{L_T^\infty L_x^r}+\left\lVert t^\gamma(u-v)\right\rVert_{L_T^\infty L_x^\infty}. \]

$\Phi$ maps $\mathcal{X}_\gamma^{s,r}(T)$ into $\mathcal{X}_\gamma^{s,r}(T)$ by Lemma $\ref{20240516local wellposedness lemma}$ and
\begin{align*}
\left\lVert\Phi u\right\rVert_{L_T^\infty H_x^{s,r}}+\left\lVert t^\gamma\Phi u\right\rVert_{L_T^\infty L_x^\infty}\lesssim&\left(1+T^{\alpha-\gamma(p-1)}\right)\left\lVert u_0\right\rVert_{L_x^q\cap H_x^{s,r}}\\
&+\left(T^{\alpha-\gamma(p-1)}+T^{\alpha-\frac{\sigma n}{r}-\gamma(p-2)}\right)\left\lVert u_0\right\rVert_{L_x^q\cap H_x^{s,r}}^p
\end{align*}
by Proposition $\ref{9.24 decay estimates2}$ and Lemmas $\ref{10.10 Main results theorem1 proof lemma2}$ and $\ref{10.9Estimates of the nonlinearity Lemma1}$. Then we can choose $T$ sufficiently small such that $\Phi$ maps $E_\gamma^{s,r}(T)$ into $E_\gamma^{s,r}(T)$.

On the other hand, for any $u,v\in E_\gamma^{s,r}(T)$, we have, by Proposition $\ref{9.24 decay estimates2}$ and Lemma $\ref{10.10 Main results theorem1 proof lemma2}$, that
\[ d_T\left(\Phi u,\Phi v\right)\lesssim\left(T^{\alpha-\gamma(p-1)}+T^{\alpha-\frac{\sigma n}{r}-\gamma(p-2)}\right)\left\lVert u_0\right\rVert_{L_x^q\cap H_x^{s,r}}^{p-1}d_T(u,v). \]

Then we can choose $T$ sufficiently small such that $\Phi$ is contraction on $E_\gamma^{s,r}(T)$. By applying the contraction mapping theorem, we can obtain the existence and uniqueness of the local solution.

\textbf{Step 2:} Let $T_{\max}$ be the supremum of all $T>0$ for which there exists a solution of $(\ref{9.5 space time fractional Schrodinger equation})$ in $\overline{\mathcal{X}_\gamma^{s,r}(T)}$ under the hypotheses in Theorem $\ref{20240515locla well-posendess}$. Lemma $\ref{20240516continuation lemma}$ shows that $T_{\max}$ exists and $0<T_{\max}\leq\infty$ and
\[ u\in C\left(\left[0,T_{\max}\right);H_x^{s,r}\right),\quad t^\gamma u\in C\left(\left[0,T_{\max}\right);L_x^\infty\right). \]

Suppose that $T_{\max}<\infty$ but both $\left\lVert u\right\rVert_{H_x^{s,r}}$ and $t^\gamma\left\lVert u(t)\right\rVert_{L_x^\infty}$ are bounded on $\left[0,T_{\max}\right]$. A direct computation leads to
\[ u\in C\left(\left[0,T_{\max}\right];H_x^{s,r}\right),\quad t^\gamma u\in C\left(\left[0,T_{\max}\right];L_x^\infty\right), \]
and hence $u$ can be extended to $\left[0,T^*\right]$ for some $T^*>T_{\max}$ by Lemma $\ref{20240516continuation lemma}$ which contradicts to the definition of $T_{\max}$. Such arguments agree that $T_{\max}<\infty$ implies that $\lim\limits_{t\uparrow T_{\max}}\left\lVert u(t)\right\rVert_{H_x^{s,r}}=\infty$ or $\lim\limits_{t\uparrow T_{\max}}t^\gamma\left\lVert u(t)\right\rVert_{L_x^\infty}$.

\textbf{Step 3:} It remains to prove the continuous dependence. Let $u\in\overline{\mathcal{X}_\gamma^{s,r}(T)}$ be the solution of $(\ref{9.5 space time fractional Schrodinger equation})$ under the hypotheses in Theorem $\ref{20240515locla well-posendess}$ with the initial data $u_0\in L_x^q\cap H_x^{s,r}$ and $u_0^k\to u_0$ in $L_x^q\cap H_x^{s,r}$ as $k\to\infty$. Let $u_k$ be the solution of $(\ref{9.5 space time fractional Schrodinger equation})$ under the hypotheses in Theorem $\ref{20240515locla well-posendess}$ with the initial data $u_0^k$. It follows from the same arguments as in Step 1 and Step 2 that
\[ u_k\in C\left(\left[0,T_{\max}^k\right);H_x^{s,r}\right),\quad t^\gamma u\in C\left(\left[0,T_{\max}^k\right);L_x^\infty\right). \]

Let $M=\left\lVert u\right\rVert_{L_T^\infty H_x^{s,r}}+\left\lVert t^\gamma u\right\rVert_{L_T^\infty L_x^\infty}$ and define
\[ T_k:=\sup\left\{t\in\left[0,T_{\max}^k\right):\left\lVert u_k(s)\right\rVert_{H_x^{s,r}}+s^\gamma\left\lVert u_k(s)\right\rVert_{L_x^\infty}\leq M\;on\;[0,t]\right\}. \]

Then when $t\leq T_k\wedge T$, by Propositions $\ref{9.24 decay estimates2}$ and $\ref{10.10 Holder type estimates}$ and Lemmas $\ref{10.10 Main results theorem1 proof lemma2}$ and $\ref{10.249.24 Main results theorem1 proof lemma3}$, there holds
\begin{align*}
&\left\lVert u_k(t)-u(t)\right\rVert_{H_x^{s,r}}+t^\gamma\left\lVert u_k(t)-u(t)\right\rVert_{L_x^\infty}\\
&\lesssim\left(1+T^{\gamma-\frac{\sigma n}{q}}\right)\left\lVert u_0^k-u_0\right\rVert_{L_x^q\cap H_x^{s,r}}\\
&+\left(1+T^\gamma\right)M^{p-1}\int_0^t\left(t-\tau\right)^{\alpha-1}\tau^{-\gamma p}\left(\left\lVert u_k(\tau)-u(\tau)\right\rVert_{H_x^{s,r}}+\tau^\gamma\left\lVert u_k(\tau)-u(\tau)\right\rVert_{L_x^\infty}\right)d\tau.
\end{align*}

Then by Lemma $\ref{10.27Gronwall inequality lemma}$, we have
\begin{align*}
\left\lVert u_k(t)-u(t)\right\rVert_{H_x^{s,r}}+&t^\gamma\left\lVert u_k(t)-u(t)\right\rVert_{L_x^\infty}\\
&\lesssim\left(1+T^{\gamma-\frac{\sigma n}{q}}\right)\left\lVert u_0^k-u_0\right\rVert_{L_x^q\cap H_x^{s,r}}\mathbb{E}_{\alpha,1-\gamma p}\left(Ct\right),
\end{align*}
for some constant $C>0$. Then
\[ \left\lVert u_k(t)-u(t)\right\rVert_{H_x^{s,r}}+t^\gamma\left\lVert u_k(t)-u(t)\right\rVert_{L_x^\infty}\leq M \]
for $k$ large enough when $t\leq T_k\wedge T$. By Lemma $\ref{20240516continuation lemma}$, $T_k>T_k\wedge T$ and hence $T_k>T$. Then $T_{\max}^k>T$. It follows that $u_k$ exists in the same space of $u$ for $k$ large enough and $u_k\to u$ in $\overline{\mathcal{X}_\gamma^{s,r}(T)}$.
\end{proof}

Similarly, we can prove Theorem $\ref{20240518locla well-posendessII}$.
\section{Proof of Theorems $\ref{20240517global solution with small data II}$ and $\ref{20240517global solution with small data III}$}
\begin{proof}[\textbf{Proof of Theorem} $\mathbf{\ref{20240517global solution with small data II}}$]
Choose $\gamma=\frac{\alpha}{p-1}$. Define a function space
\[ \mathcal{Y}_\gamma^{s,r}:=\left\{u\in L_t^\infty H_x^{s,r},\;t^\gamma u\in L_t^\infty L_x^\infty:\left\lVert u\right\rVert_{L_t^\infty H_x^{s,r}}+\left\lVert t^\gamma u\right\rVert_{L_t^\infty L_x^\infty}\leq\left\lVert u_0\right\rVert_{L_x^{\frac{n(p-1)}{2\beta}}}+\left\lVert u_0\right\rVert_{H_x^{s,r}}\right\} \]
with its metric
\[ d_{\mathcal{Y}}(u,v)=\left\lVert u-v\right\rVert_{L_t^\infty L_x^r}+\left\lVert t^\gamma(u-v)\right\rVert_{L_t^\infty L_x^\infty},\quad u,v\in\mathcal{Y}_\gamma^{s,r}. \]

Let $\Phi u=S_tu_0+i\mathcal{M}F(u)$. It suffices to prove that $\Phi$ has a fixed point in $\mathcal{Y}_\gamma^{s,r}$. By Propositions $\ref{9.24 decay estimates2}$ and Lemmas $\ref{10.10 Main results theorem1 proof lemma2}$ and $\ref{10.9Estimates of the nonlinearity Lemma1}$, we have
\[ \left\lVert\Phi u\right\rVert_{H_x^{s,r}}+t^\gamma\left\lVert\Phi u\right\rVert_{L_x^\infty}\lesssim\delta+\delta^p. \]

$\Phi$ maps $\mathcal{Y}_\gamma^{s,r}$ into $\mathcal{Y}_\gamma^{s,r}$ since $\delta$ is sufficiently small. On the other hand, for any $u,v\in\mathcal{Y}_\gamma^{s,r}$, by Proposition $\ref{10.10 Holder type estimates}$ and Lemma $\ref{10.10 Main results theorem1 proof lemma2}$, we have
\[ d_{\mathcal{Y}}\left(\Phi u,\Phi v\right)\lesssim\delta^{p-1}d_{\mathcal{Y}}(u,v). \]

Then $\Phi$ is a contraction on $\mathcal{Y}_\gamma^{s,r}$ since $\delta$ sufficiently small. Applying the contraction mapping theorem we can complete the result.
\end{proof}

The proof of Theorem $\ref{20240517global solution with small data III}$ is similar and we omit it.
\section{Proof of Theorems $\ref{20240517global solution arbitrary data}$ and $\ref{20240519global solution arbitrary dataII}$}
\begin{proof}[\textbf{Proof of Theorem} $\mathbf{\ref{20240517global solution arbitrary data}}$]
It suffices to prove that $\left\lVert u(t)\right\rVert_{H_x^{s,r}}$ and $t^\gamma\left\lVert u(t)\right\rVert_{L_x^\infty}$ are bounded on every finite time interval $[0,T]$ by Theorem $\ref{20240515locla well-posendess}$.

\textbf{Step 1:} By Propositions $\ref{9.24 decay estimates2}$ and $\ref{10.10 Holder type estimates}$, Lemma $\ref{10.8A priori estimates2}$ and Sobolev's embedding, it follows that
\begin{align*}
t^\gamma\left\lVert u(t)\right\rVert_{L_x^\infty}&\lesssim t^{\gamma-\frac{\sigma n}{q}}\left\lVert u_0\right\rVert_{L_x^q}+t^\gamma\int_0^t\left(t-\tau\right)^{\alpha-\sigma n-1}\left\lVert u(\tau)\right\rVert_{L_x^p}^pd\tau\\
&\lesssim t^{\gamma-\frac{\sigma n}{q}}\left\lVert u_0\right\rVert_{L_x^\infty}+t^{\alpha-\sigma n+\gamma}E(u_0)^p,
\end{align*}
which bounds $t^\gamma\left\lVert u(t)\right\rVert_{L_x^\infty}$ on $[0,T]$.

\textbf{Step 2:} By Propositions $\ref{9.24 decay estimates2}$ and $\ref{10.10 Holder type estimates}$ and Lemma $\ref{10.9Estimates of the nonlinearity Lemma1}$, we have
\[ \left\lVert u(t)\right\rVert_{H_x^{s,r}}\lesssim\left\lVert u_0\right\rVert_{H_x^{s,r}}+\left\lVert t^\gamma u\right\rVert_{L_T^\infty L_x^\infty}^{p-1}\int_0^t\left(t-\tau\right)^{\alpha-1}\tau^{-\gamma(p-1)}\left\lVert u(\tau)\right\rVert_{H_x^{s,r}}d\tau. \]

Hence by Lemma $\ref{10.27Gronwall inequality lemma}$, it follows that
\[ \left\lVert u\right\rVert_{L_T^\infty H_x^{s,r}}\lesssim\left\lVert u_0\right\rVert_{H_x^{s,r}}\mathbb{E}_{\alpha,1-\gamma(p-1)}\left(CT\right) \]
for some constant $C>0$.
\end{proof}

Similarly, we can prove Theorem $\ref{20240519global solution arbitrary dataII}$.
\section{Conclusion}
It's well-known that space-time fractional Schr\"odinger equation displays a nonlocal behavior both in space and time and plays an important role in fractional quantum mechanics. However, the theoretical studies on it are still rudimentary. In this paper we present some new results on the local and global well-posedess of the space-time fractioal Schr\"odinger equation which are generalizations of previous work. To do this, we first acquire the decay estimates, H\"older type estimates of the evolution operators and a priori estimates which are elementary estimates in the research. After that we prove the results by some technical operations which, we believe, can also be used in the research for some other equations such as Navier-Stokes equation and Rayleigh-Stokes equation.
\section*{Data availability statement}
\noindent Data openly available in a public repository.
\section*{Declarations}
\noindent\textbf{Conflict of interest}\quad The authors declare that they have no conflict of interest.
\appendix
\section{The derivation of $(\ref{9.5 The equivalent integral equation equation1})$}
\label{20240521The derivation of integral equation}
In this section, we provide the derivation of the integral equation $(\ref{9.5 The equivalent integral equation equation1})$. Precisely, we will prove the following lemma.
\begin{lemma}
Let $u$ be the solution of $(\ref{9.5 space time fractional Schrodinger equation})$. Then $u$ satisfies $(\ref{9.5 The equivalent integral equation equation1})$.
\label{20240430equivalence}
\end{lemma}
\begin{proof}[Proof]
Applying the Fourier transform to $(\ref{9.5 space time fractional Schrodinger equation})$, we can obtain
\begin{equation}
\label{20240430Fourier transform original equation}
\begin{cases}
i\partial_t^\alpha\widehat{u}-\left\lvert\xi\right\rvert^{2\beta}\widehat{u}+\left(F(u)\right)^\wedge,\quad&\xi\in\mathbb{R}^n,\;t>0,\\
\widehat{u}(0,\xi)=\widehat{u}_0(\xi),\quad&\xi\in\mathbb{R}^n.
\end{cases}
\end{equation}

By the theory of fractional ordinary differential equations{\cite[Proposition 5.10]{Theory-and-Applications-of-Fractional-Differential-Equations}}, the solution of $(\ref{20240430Fourier transform original equation})$ can be written as
\begin{equation}
\label{20240430Fourier transform original equation ODE}
\widehat{u}(t,\xi)=a_t(\xi)\widehat{u}_0(\xi)+i\int_0^tb_{t-\tau}(\xi)\left(F(u(\tau))\right)^\wedge d\tau.
\end{equation}

Then we can apply the inverse Fourier transform to $(\ref{20240430Fourier transform original equation ODE})$ to deduce $(\ref{9.5 The equivalent integral equation equation1})$.
\end{proof}
\begin{remark}
$(\ref{9.5 space time fractional Schrodinger equation})$ and $(\ref{9.5 The equivalent integral equation equation1})$ are equivalent in the distributional sense.
\end{remark}
\section{The derivation of $(\ref{20240503ML Mainardi function})$}
\label{20240503The derivation}
\begin{definition}[{\cite[(F.13)]{Fractional-Calculus-and-Waves-in-Linear-Viscoelasticity}}]
\label{20240503Definition Mainardi}
Let $0<\nu<1$, $z\in\mathbb{C}$. The Mainardi function $M_\nu(z)$ is given by
\[ M_\nu(z):=\sum\limits_{k=0}^\infty\frac{\left(-z\right)^k}{k!\Gamma\left(-\nu k+(1-\nu)\right)}. \]
\end{definition}
Recall that
\begin{equation}
\label{20240503Mainardi function integral}
\int_0^\infty\theta^\delta M_\nu(\theta)d\theta=\frac{\Gamma(\delta+1)}{\Gamma(\nu\delta+1)},\quad\delta>-1,\quad0\leq\nu<1.
\end{equation}
See {\cite[(F.33)]{Fractional-Calculus-and-Waves-in-Linear-Viscoelasticity}} for detailed proof. In addition, the following relationships between Mittag-Leffler function and Mainardi function are true:
\begin{equation}
\label{20240503relationship MF ML}
E_{\alpha,1}(-z)=\int_0^\infty M_\alpha(\theta)e^{-z\theta}d\theta,\quad E_{\alpha,\alpha}(-z)=\int_0^\infty\alpha\theta M_\alpha(\theta)e^{-z\theta}d\theta,\quad z\in\mathbb{C}.
\end{equation}
\begin{proof}[The derivation of $(\ref{20240503ML Mainardi function})$]
Using $(\ref{20240503relationship MF ML})$, we have
\begin{equation}
\label{20240503The derivation proof}
a_t(\xi)=\int_0^\infty M_\alpha(\theta)e^{-i\left\lvert\xi\right\rvert^{2\beta}t^\alpha\theta}d\theta,\quad b_t(\xi)=\int_0^\infty\alpha\theta M_\alpha(\theta)t^{\alpha-1}e^{-i\left\lvert\xi\right\rvert^{2\beta}t^\alpha\theta}d\theta.
\end{equation}

Using the fractional Schr\"odinger operator $e^{it\left(-\Delta\right)^\beta}$, i.e., $e^{it\left(-\Delta\right)^\beta}\phi=\mathscr{F}^{-1}\left(e^{it\left\lvert\xi\right\rvert^{2\beta}\widehat{\phi}}\right)$, we obtain $(\ref{20240503ML Mainardi function})$.
\end{proof}
\section{The derivation of $(\ref{9.23MLinequalities hold})$}
\label{20240502The derivation}
Define a piecewise function $\phi(x)$ as
\[ \phi(x)=\begin{cases}
e^{-\frac{1}{x^2}},\quad&x>0,\\
0,\quad&x\leq0.
\end{cases} \]
and truncated functions $\chi_1(t)$ and $\chi_1^c(t)$ as
\[ \chi_1(t)=\frac{\phi(2-t)}{\phi(2-t)+\phi(t-1)},\;\chi_1^c(t)=1-\chi_1(t),\quad\rm{for}\;t\geq0. \]

Note that $\chi_1(t)$ is smooth and $\chi_1(t)\equiv1$ for $0\leq t\leq1$ and $\chi_1(t)\equiv0$ for $t\geq2$. Then for a given large enough constant $M$, we can define the radial functions $\chi_t(\xi)$, $\chi_t^c(\xi)$ as $\chi_t(\xi):=\chi_1\left(\frac{t^\sigma\left\lvert\xi\right\rvert}{M}\right)$, $\chi_t^c(\xi):=\chi_1^c\left(\frac{t^\sigma\left\lvert\xi\right\rvert}{M}\right)$ respectively.
\begin{remark}
\label{20240502appendix truncated function bounded}
$\left\lvert\chi_1(t)\right\rvert\leq1$ and hence $\left\lvert\chi_t(\xi)\right\rvert,\left\lvert\chi_t^c(\xi)\right\rvert\leq1$. Additionally, $\left\lvert\chi_1'(t)\right\rvert$ is bounded where $\chi_1'(t)$ is the first derivative of $\chi_1(t)$ with respect to $t$.
\end{remark}
\begin{proof}[The derivation of $(\ref{9.23MLinequalities hold})$]
Define the following operators as
\begin{align*}
&\chi_t(D)\phi=\mathscr{F}^{-1}\left(\chi_t(\xi)\widehat{\phi}\right),\quad\chi_t^c(D)\phi=\mathscr{F}^{-1}\left(\chi_t^c(\xi)\widehat{\phi}\right)\\
&\left\lvert\nabla\right\rvert^\theta\phi=\mathscr{F}^{-1}\left(\left\lvert\xi\right\rvert^\theta\widehat{\phi}\right),\quad O\left(\left\lvert\nabla\right\rvert^\theta\right)\phi=\mathscr{F}^{-1}\left(O\left(\left\lvert\xi\right\rvert^\theta\right)\widehat{\phi}\right).
\end{align*}

Using the asymptotic expansion of the Mittag-Leffler function{\cite[Theorem 1.4]{Fractional-differential-equations-an-introduction-to-fractional-derivatives-fractional-differential-equations-to-methods-of-their-solution-and-some-of-their-applications}}, we can write $a_t(\xi)$, $b_t(\xi)$ as
\begin{align*}
&a_t(\xi)=-\frac{i}{\Gamma(1-\alpha)}t^{-\alpha}\left\lvert\xi\right\rvert^{-2\beta}+t^{-2\alpha}O\left(\left\lvert\xi\right\rvert^{-4\beta}\right),\\
&b_t(\xi)=\frac{1}{\Gamma(-\alpha)}t^{-\alpha-1}\left\lvert\xi\right\rvert^{-4\beta}+t^{-2\alpha-1}O\left(\left\lvert\xi\right\rvert^{-6\beta}\right),
\end{align*}
respectively, for $t^\sigma\left\lvert\xi\right\rvert$ large enough. Then for a given large enough constant $M$, we have
\begin{align*}
S_t\phi&=\mathscr{F}^{-1}\left(a_t(\xi)\widehat{\phi}\right)\\
&=\mathscr{F}^{-1}\left(a_t(\xi)\chi_t(\xi)\widehat{\phi}\right)+\mathscr{F}^{-1}\left(a_t(\xi)\chi_t^c(\xi)\widehat{\phi}\right)\\
&=\mathscr{F}^{-1}\left(a_t(\xi)\chi_t(\xi)\widehat{\phi}\right)-\frac{i}{\Gamma(1-\alpha)}t^{-\alpha}\mathscr{F}^{-1}\left(\left\lvert\xi\right\rvert^{-2\beta}\chi_t^c(\xi)\widehat{\phi}\right)+t^{-2\alpha}\mathscr{F}^{-1}\left(O\left(\left\lvert\xi\right\rvert^{-4\beta}\right)\chi_t^c(\xi)\widehat{\phi}\right)\\
&=S_t\chi_t(D)\phi-\frac{i}{\Gamma(1-\alpha)}t^{-\alpha}\left\lvert\nabla\right\rvert^{-2\beta}\chi_t^c(D)\phi+t^{-2\alpha}O\left(\left\lvert\nabla\right\rvert^{-4\beta}\right)\chi_t^c(D)\phi,
\end{align*}
and
\begin{align*}
P_t\phi&=\mathscr{F}^{-1}\left(b_t(\xi)\widehat{\phi}\right)\\
&=\mathscr{F}^{-1}\left(b_t(\xi)\chi_t(\xi)\widehat{\phi}\right)+\mathscr{F}^{-1}\left(b_t(\xi)\chi_t^c(\xi)\widehat{\phi}\right)\\
&=\mathscr{F}^{-1}\left(b_t(\xi)\chi_t(\xi)\widehat{\phi}\right)+\frac{1}{\Gamma(-\alpha)}t^{-\alpha-1}\mathscr{F}^{-1}\left(\left\lvert\xi\right\rvert^{-4\beta}\chi_t^c(\xi)\widehat{\phi}\right)+t^{-2\alpha-1}\mathscr{F}^{-1}\left(O\left(\left\lvert\xi\right\rvert^{-6\beta}\right)\chi_t^c(\xi)\widehat{\phi}\right)\\
&=P_t\chi_t(D)\phi+\frac{1}{\Gamma(-\alpha)}t^{-\alpha-1}\left\lvert\nabla\right\rvert^{-4\beta}\chi_t^c(D)\phi+t^{-2\alpha-1}O\left(\left\lvert\nabla\right\rvert^{-6\beta}\right)\chi_t^c(D)\phi.
\end{align*}
\end{proof}

Let $a(t)=E_{\alpha,1}(t)$, $b(t)=E_{\alpha,\alpha}(t)$, and
\begin{align*}
&K^\delta[a](x)=\int_{\mathbb{R}^n}e^{ix\cdot\xi}\left\lvert\xi\right\rvert^\delta a\left(-i\left\lvert\xi\right\rvert^{2\beta}\right)\chi_1\left(\frac{\left\lvert\xi\right\rvert}{M}\right)d\xi,\\
&K_\delta(x)=\int_{\mathbb{R}^n}e^{ix\cdot\xi}\left\lvert\xi\right\rvert^\delta\chi_1^c\left(\frac{\left\lvert\xi\right\rvert}{M}\right)d\xi.
\end{align*}
\begin{remark}
Note that by the aysymptotic expansion of the Mittag-Leffler function, $\left\lvert a(it)\right\rvert\lesssim t^{-1}$ and so does $b(it)$.
\end{remark}
\begin{remark}
It is easy to verify that $K[a](x)$ and $K[b](x)$ are bounded for $\theta\geq0$, and if $\delta<-n$, $K_\delta(x)$ is bounded.
\end{remark}
\begin{lemma}
\label{20240514KaKb estimate}
Let $0\leq\theta<2\beta-\frac{n+1}{2}$ for $\beta>\frac{n+1}{4}$. Then for any $\delta\geq0$, $K^{\theta+\delta}[a](x)$ has the following pointwise estimate
\[ \left\lvert K^{\theta+\delta}[a](x)\right\rvert\lesssim\left(1+\left\lvert x\right\rvert\right)^{-n-1}, \]
and so does $K^{\theta+\delta}[b](x)$.
\end{lemma}
\begin{proof}[Proof]
By the Fourier transform of surface measure{\cite[B.4]{Classical-fourier-analysis}} and the asymptotic behavior of the Bessel function{\cite[B.6 and B.7]{Classical-fourier-analysis}}, we have
\begin{align*}
\left\lvert K^{\theta+\delta}[a](x)\right\rvert&=\left\lvert\int_0^\infty\int_{\left\lvert\xi\right\rvert=r}e^{ix\cdot\xi}r^{\theta+\delta} a\left(-ir^{2\beta}\right)\chi_1\left(\frac{r}{M}\right)d\xi dr\right\rvert\\
&=\left\lvert\int_0^\infty r^{\theta+\delta+\frac{n}{2}}a\left(-ir^{2\beta}\right)\chi_1\left(\frac{r}{M}\right)J_{\frac{n-2}{2}}\left(r\left\lvert x\right\rvert\right)dr\left\lvert x\right\rvert^{1-\frac{n}{2}}\right\rvert\\
&\lesssim\int_0^1r^{\theta+n-1}dr+\int_1^{2M}r^{\theta-2\beta+\frac{n-1}{2}}dr\left\lvert x\right\rvert^{\frac{1-n}{2}}\\
&\lesssim M^{-\frac{(n+1)(4\beta-2\theta-n-1)}{n+3}}+M^{\theta-2\beta+\frac{n+1}{2}}\left\lvert x\right\rvert^{\frac{1-n}{2}}.
\end{align*}

Taking $M$ large enough such that $M>\left\lvert x\right\rvert^{\frac{n+3}{4\beta-2\theta-n-1}}$, we have
\[ \left\lvert K[a](x)\right\rvert\lesssim\left\lvert x\right\rvert^{-n-1}. \]

This together with the boundedness of $K[a](x)$ completes the proof.
\end{proof}
\begin{lemma}
\label{20240514Kdelta estimate}
Let $\delta<-n$. $K_\delta(x)$ has the following pointwise estimate
\[ \left\lvert K_\delta(x)\right\rvert\lesssim\left(1+\left\lvert x\right\rvert\right)^{-n-1}. \]
\end{lemma}
\begin{proof}[Proof]
Arguing similarly as the proof of Lemma $\ref{20240514KaKb estimate}$ we have
\begin{align*}
\left\lvert K_\delta(x)\right\rvert&\lesssim\int_M^\infty r^{\delta+\frac{n-1}{2}-1}dr\left\lvert x\right\rvert^{\frac{1-n}{2}}\\
&\lesssim M^{\delta+\frac{n-1}{2}}\left\lvert x\right\rvert^{\frac{1-n}{2}}.
\end{align*}

Taking $M$ large enough such that $M>\frac{n+3}{1-n-2\delta}$, we have
\[ \left\lvert K_\delta(x)\right\rvert\lesssim\left\lvert x\right\rvert^{-n-1}. \]

Due to the boundedness of $K_\delta(x)$, we can compelte the proof.
\end{proof}
\begin{lemma}
\label{20240505Decay estimates St first}
Let $0\leq\theta<2\beta-\frac{n+1}{2}$ for $\beta>\frac{n+1}{4}$ and $1\leq q\leq r\leq\infty$. Then
\begin{align*}
&\left\lVert\left\lvert\nabla\right\rvert^\theta S_t\chi_t(D)\phi\right\rVert_{L_x^r}\lesssim t^{-\sigma\theta-\sigma n\left(\frac{1}{q}-\frac{1}{r}\right)}\left\lVert\phi\right\rVert_{L_x^q},\\
&\left\lVert\left\lvert\nabla\right\rvert^\theta P_t\chi_t(D)\phi\right\rVert_{L_x^r}\lesssim t^{-\sigma\theta-\sigma n\left(\frac{1}{q}-\frac{1}{r}\right)+\alpha-1}\left\lVert\phi\right\rVert_{L_x^q}.
\end{align*}
\end{lemma}
\begin{proof}[Proof]
By scaling, we have
\begin{align*}
\left\lvert\nabla\right\rvert^\theta S_t\chi_t(D)\phi&=\mathscr{F}^{-1}\left(\left\lvert\xi\right\rvert^\theta a_t(\xi)\chi_t(\xi)\right)*\phi\\
&\sim t^{-\sigma(\theta+n)}K[a]\left(t^{-\sigma}\cdot\right)*\phi.
\end{align*}

By Young's inequality{\cite[Lemma 1.4]{Fourier-Analysis-and-Nonlinear-Partial-Differential-Equations}} and Lemma $\ref{20240514KaKb estimate}$, we have
\begin{align*}
\left\lVert\left\lvert\nabla\right\rvert^\theta S_t\chi_t(D)\phi\right\rVert_{L_x^r}&\lesssim t^{-\sigma(\theta+n)}\left\lVert K[a]\left(t^{-\sigma}\cdot\right)\right\rVert_{L_x^s}\left\lVert\phi\right\rVert_{L_x^q},\\
&\lesssim t^{-\sigma\theta-\sigma n\left(\frac{1}{q}-\frac{1}{r}\right)}\left\lVert\phi\right\rVert_{L_x^q},
\end{align*}
where $\frac{1}{r}+1=\frac{1}{s}+\frac{1}{q}$.

The proof for $\left\lvert\nabla\right\rvert^\theta P_t\chi_t(D)\phi$ is similar.
\end{proof}
\begin{lemma}
\label{20240505Decay estimates St second}
Let $\theta<2\beta-n$ for $\beta>\frac{n}{2}$ and $1\leq q\leq r\leq\infty$. Then
\begin{align*}
&\left\lVert\left\lvert\nabla\right\rvert^{\theta-2\beta}\chi_t^c(D)\phi\right\rVert_{L_x^r}\lesssim t^{\alpha-\sigma\theta-\sigma n\left(\frac{1}{q}-\frac{1}{r}\right)}\left\lVert\phi\right\rVert_{L_x^q},\\
&\left\lVert\left\lvert\nabla\right\rvert^{\theta-4\beta}\chi_t^c(D)\phi\right\rVert_{L_x^r}\lesssim t^{2\alpha-\sigma\theta-\sigma n\left(\frac{1}{q}-\frac{1}{r}\right)}\left\lVert\phi\right\rVert_{L_x^q},
\end{align*}
\end{lemma}
\begin{proof}[Proof]
By scaling, we have
\begin{align*}
&\left\lvert\nabla\right\rvert^{\theta-2\beta}\chi_t^c(D)\phi\sim t^{\alpha-\sigma(\theta+n)}K_{\theta-2\beta}\left(t^{-\sigma}\cdot\right)*\phi,\\
&\left\lvert\nabla\right\rvert^{\theta-4\beta}\chi_t^c(D)\phi\sim t^{2\alpha-\sigma(\theta+n)}K_{\theta-4\beta}\left(t^{-\sigma}\cdot\right)*\phi.
\end{align*}

By Young's inequality and Lemma $\ref{20240514Kdelta estimate}$, we can complete the proof.
\end{proof}

Similar to Lemma $\ref{20240505Decay estimates St second}$, we have the following lemma.
\begin{lemma}
\label{20240505Decay estimates St third}
Let $\theta<4\beta-n$ for $\beta>\frac{n}{4}$ and $1\leq q\leq r\leq\infty$. Then
\begin{align*}
&\left\lVert\left\lvert\nabla\right\rvert^\theta O\left(\left\lvert\nabla\right\rvert^{-4\beta}\right)\chi_t^c(D)\phi\right\rVert_{L_x^r}\lesssim t^{2\alpha-\sigma\theta-\sigma n\left(\frac{1}{q}-\frac{1}{r}\right)}\left\lVert\phi\right\rVert_{L_x^q},\\
&\left\lVert\left\lvert\nabla\right\rvert^\theta O\left(\left\lvert\nabla\right\rvert^{-6\beta}\right)\chi_t^c(D)\phi\right\rVert_{L_x^r}\lesssim t^{3\alpha-\sigma\theta-\sigma n\left(\frac{1}{q}-\frac{1}{r}\right)}\left\lVert\phi\right\rVert_{L_x^q}.
\end{align*}
\end{lemma}
\begin{lemma}
\label{20240505Holder estimates St first}
Let $0\leq\theta<2\beta-n$ for $\beta>\frac{n}{2}$ and $1\leq q\leq r\leq\infty$. Then for any $t_1,t_2>0$, we have
\begin{align*}
&\begin{aligned}
\left\lVert\left\lvert\nabla\right\rvert^\theta\left(S_{t_1}\chi_{t_1}(D)-\right.\right.&\left.\left.S_{t_2}\chi_{t_2}(D)\right)\phi\right\rVert_{L_x^r}\\
&\lesssim\left(t_1\wedge t_2\right)^{-1}\left\lvert t_1^{1-\sigma\theta-\sigma n\left(\frac{1}{q}-\frac{1}{r}\right)}-t_2^{1-\sigma\theta-\sigma n\left(\frac{1}{q}-\frac{1}{r}\right)}\right\rvert\left\lVert\phi\right\rVert_{L_x^q},
\end{aligned}\\
&\begin{aligned}
\left\lVert\left\lvert\nabla\right\rvert^\theta\left(P_{t_1}\chi_{t_1}(D)-\right.\right.&\left.\left.P_{t_2}\chi_{t_2}(D)\right)\phi\right\rVert_{L_x^r}\\
&\lesssim\left\lvert t_1^{\alpha-1-\sigma\theta-\sigma n\left(\frac{1}{q}-\frac{1}{r}\right)}-t_2^{\alpha-1-\sigma\theta-\sigma n\left(\frac{1}{q}-\frac{1}{r}\right)}\right\rvert\left\lVert\phi\right\rVert_{L_x^q}.
\end{aligned}
\end{align*}
\end{lemma}
\begin{proof}[Proof]
Assume $t_1>t_2>0$ without loss of generality. Note that $\frac{d}{dt}a_t(\xi)=-i\left\lvert\xi\right\rvert^{2\beta}b_t(\xi)$. A direct computation leads to
\begin{align*}
&\mathscr{F}^{-1}\left(\left\lvert\xi\right\rvert^\theta\left(a_{t_1}(\xi)\chi_{t_1}(\xi)-a_{t_2}(\xi)\chi_{t_2}(\xi)\right)\right)\\
&=\mathscr{F}^{-1}\left(\int_{t_2}^{t_1}-i\left\lvert\xi\right\rvert^{2\beta}b_\tau(\xi)\chi_\tau(\xi)d\tau\right)+\mathscr{F}^{-1}\left(\int_{t_2}^{t_1}a_\tau(\xi)\frac{\sigma\tau^{\sigma-1}\left\lvert\xi\right\rvert}{M}\chi_1'\left(\frac{\tau^\sigma\left\lvert\xi\right\rvert}{M}\right)d\tau\right)\\
&=:I_1(t_1,t_2,x)+I_2(t_1,t_2,x).
\end{align*}

By scaling, we have
\[ I_1(t_1,t_2,x)=-i\int_{t_2}^{t_1}\tau^{-\sigma(\theta+n)-1}K^{\theta+2\beta}[b]\left(\tau^{-\sigma}x\right)d\tau. \]

Using Lemma $\ref{20240514KaKb estimate}$, we obtain
\[ \left\lvert I_1(t_1,t_2,x)\right\rvert\lesssim\int_{t_2}^{t_1}\tau^{-\sigma(\theta+n)-1}\left(1+\tau^{-\sigma}\left\lvert x\right\rvert\right)^{-n-1}\tau, \]
and hence
\[ \left\lVert I_1(t_1,t_2,\cdot)\right\rVert_{L_x^s}\lesssim\int_{t_2}^{t_1}\tau^{-\sigma\theta-\sigma n\left(\frac{1}{q}-\frac{1}{r}\right)-1}d\tau\lesssim t_2^{-1}\left(t_1^{1-\sigma\theta-\sigma n\left(\frac{1}{q}-\frac{1}{r}\right)}-t_2^{1-\sigma\theta-\sigma n\left(\frac{1}{q}-\frac{1}{r}\right)}\right), \]
where $\frac{1}{r}+1=\frac{1}{s}+\frac{1}{q}$. Similarly arguing as above, we have
\[ \left\lVert I_2(t_1,t_2,\cdot)\right\rVert_{L_x^s}\lesssim t_2^{-1}\left(t_1^{1-\sigma\theta-\sigma n\left(\frac{1}{q}-\frac{1}{r}\right)}-t_2^{1-\sigma\theta-\sigma n\left(\frac{1}{q}-\frac{1}{r}\right)}\right). \]

By Young's inequality, we obtain
\begin{align*}
\left\lVert\left\lvert\nabla\right\rvert^\theta\left(S_{t_1}\chi_{t_1}(D)-S_{t_2}\chi_{t_2}(D)\right)\phi\right\rVert_{L_x^r}&\leq\left\lVert I_1(t_1,t_2,\cdot)*\phi\right\rVert_{L_x^r}+\left\lVert I_2(t_1,t_2,\cdot)*\phi\right\rVert_{L_x^r}\\
&\lesssim t_2^{-1}\left(t_1^{1-\sigma\theta-\sigma n\left(\frac{1}{q}-\frac{1}{r}\right)}-t_2^{1-\sigma\theta-\sigma n\left(\frac{1}{q}-\frac{1}{r}\right)}\right)\left\lVert\phi\right\rVert_{L_x^q}.
\end{align*}

The proof for $\left\lvert\nabla\right\rvert^\theta\left(P_{t_1}\chi_{t_1}(D)-P_{t_2}\chi_{t_2}(D)\right)\phi$ is similar.
\end{proof}
\begin{lemma}
\label{20240505Holder estimates St second}
Let $0\leq\theta<2\beta-n$ for $\beta>\frac{n}{2}$ and $1\leq q\leq r\leq\infty$. Then for any $t_1,t_2>0$, we have
\begin{align*}
&\begin{aligned}
\left\lVert\left\lvert\nabla\right\rvert^{\theta-2\beta}\left(t_1^{-\alpha}\chi_{t_1}^c(D)-t_2^{-\alpha}\right.\right.&\left.\left.\chi_{t_2}^c(D)\right)\phi\right\rVert_{L_x^r}\\
&\lesssim\left(t_1\wedge t_2\right)^{-1}\left\lvert t_1^{1-\sigma\theta-\sigma n\left(\frac{1}{q}-\frac{1}{r}\right)}-t_2^{1-\sigma\theta-\sigma n\left(\frac{1}{q}-\frac{1}{r}\right)}\right\rvert\left\lVert\phi\right\rVert_{L_x^q},
\end{aligned}\\
&\begin{aligned}
\left\lVert\left\lvert\nabla\right\rvert^{\theta-4\beta}\left(t_1^{-\alpha-1}\chi_{t_1}^c(D)-t_2^{-\alpha-1}\right.\right.&\left.\left.\chi_{t_2}^c(D)\right)\phi\right\rVert_{L_x^r}\\
&\lesssim\left\lvert t_1^{\alpha-1-\sigma\theta-\sigma n\left(\frac{1}{q}-\frac{1}{r}\right)}-t_2^{\alpha-1-\sigma\theta-\sigma n\left(\frac{1}{q}-\frac{1}{r}\right)}\right\rvert\left\lVert\phi\right\rVert_{L_x^q}.
\end{aligned}
\end{align*}
\end{lemma}
\begin{proof}[Proof]
Assume $t_1>t_2>0$ without loss of generality. By a direct computation, we have
\begin{align*}
&\mathscr{F}^{-1}\left(\left\lvert\xi\right\rvert^{\theta-2\beta}\left(t_1^{-\alpha}\chi_{t_1}^c(\xi)-t_2^{-\alpha}\chi_{t_2}^c(\xi)\right)\right)\\
&=\mathscr{F}^{-1}\left(\int_{t_2}^{t_1}-\alpha\tau^{-\alpha-1}\chi_\tau^c(\xi)\right)+\mathscr{F}^{-1}\left(\int_{t_2}^{t_1}\tau^{-\alpha}\frac{\sigma\tau^{\sigma-1}\left\lvert\xi\right\rvert}{M}\chi_1'\left(\frac{\tau^\sigma\left\lvert\xi\right\rvert}{M}\right)d\tau\right)\\
&=:I_3(t_1,t_2,x)+I_4(t_1,t_2,x).
\end{align*}

By scaling, we have
\[ \left\lvert I_3(t_1,t_2,x)\right\rvert\sim\left\lvert\int_{t_2}^{t_1}\tau^{-\sigma(\theta+n)-1}K_{\theta-2\beta}\left(\tau^{-\sigma}x\right)d\tau\right\rvert, \]
and hence, by Lemma $\ref{20240514Kdelta estimate}$,
\[ \left\lVert I_3(t_1,t_2,\cdot)\right\rVert_{L_x^s}\lesssim t_2^{-1}\left(t_1^{1-\sigma\theta-\sigma n\left(\frac{1}{q}-\frac{1}{r}\right)}-t_2^{1-\sigma\theta-\sigma n\left(\frac{1}{q}-\frac{1}{r}\right)}\right), \]
where $\frac{1}{r}+1=\frac{1}{s}+\frac{1}{q}$. Similarly, we have
\[ \left\lVert I_4(t_1,t_2,\cdot)\right\rVert_{L_x^s}\lesssim t_2^{-1}\left(t_1^{1-\sigma\theta-\sigma n\left(\frac{1}{q}-\frac{1}{r}\right)}-t_2^{1-\sigma\theta-\sigma n\left(\frac{1}{q}-\frac{1}{r}\right)}\right). \]

Then we can complete the proof for $\left\lvert\nabla\right\rvert^{\theta-2\beta}\left(t_1^{-\alpha}\chi_{t_1}^c(D)-t_2^{-\alpha}\chi_{t_2}^c(D)\right)\phi$ by Young's inequality. The proof for $\left\lvert\nabla\right\rvert^{\theta-4\beta}\left(t_1^{-\alpha-1}\chi_{t_1}^c(D)-t_2^{-\alpha-1}\chi_{t_2}^c(D)\right)\phi$ is analogous.
\end{proof}

Arguing similarly as in Lemma $\ref{20240505Holder estimates St second}$, the following lemma holds.
\begin{lemma}
\label{20240505Holder estimates St third}
Let $0\leq\theta<2\beta-n$ for $\beta>\frac{n}{2}$ and $1\leq q\leq r\leq\infty$. Then for any $t_1,t_2>0$, we have
\begin{align*}
&\begin{aligned}
\left\lVert\left\lvert\nabla\right\rvert^\theta O\left(\left\lvert\nabla\right\rvert^{-4\beta}\right)\left(t_1^{-2\alpha}\chi_{t_1}^c(D)-\right.\right.&\left.\left.t_2^{-2\alpha}\chi_{t_2}^c(D)\right)\phi\right\rVert_{L_x^r}\\
&\lesssim\left(t_1\wedge t_2\right)^{-1}\left\lvert t_1^{1-\sigma\theta-\sigma n\left(\frac{1}{q}-\frac{1}{r}\right)}-t_2^{1-\sigma\theta-\sigma n\left(\frac{1}{q}-\frac{1}{r}\right)}\right\rvert\left\lVert\phi\right\rVert_{L_x^q},
\end{aligned}\\
&\begin{aligned}
\left\lVert\left\lvert\nabla\right\rvert^\theta O\left(\left\lvert\nabla\right\rvert^{-6\beta}\right)\left(t_1^{-2\alpha-1}\chi_{t_1}^c(D)-\right.\right.&\left.\left.t_2^{-2\alpha-1}\chi_{t_2}^c(D)\right)\phi\right\rVert_{L_x^r}\\
&\lesssim\left\lvert t_1^{\alpha-1-\sigma\theta-\sigma n\left(\frac{1}{q}-\frac{1}{r}\right)}-t_2^{\alpha-1-\sigma\theta-\sigma n\left(\frac{1}{q}-\frac{1}{r}\right)}\right\rvert\left\lVert\phi\right\rVert_{L_x^q}.
\end{aligned}
\end{align*}
\end{lemma}
\bibliography{Reference}


\begin{thebibliography}{43}
\ifx \bisbn   \undefined \def \bisbn  #1{ISBN #1}\fi
\ifx \binits  \undefined \def \binits#1{#1}\fi
\ifx \bauthor  \undefined \def \bauthor#1{#1}\fi
\ifx \batitle  \undefined \def \batitle#1{#1}\fi
\ifx \bjtitle  \undefined \def \bjtitle#1{#1}\fi
\ifx \bvolume  \undefined \def \bvolume#1{\textbf{#1}}\fi
\ifx \byear  \undefined \def \byear#1{#1}\fi
\ifx \bissue  \undefined \def \bissue#1{#1}\fi
\ifx \bfpage  \undefined \def \bfpage#1{#1}\fi
\ifx \blpage  \undefined \def \blpage #1{#1}\fi
\ifx \burl  \undefined \def \burl#1{\textsf{#1}}\fi
\ifx \doiurl  \undefined \def \doiurl#1{\url{https://doi.org/#1}}\fi
\ifx \betal  \undefined \def \betal{\textit{et al.}}\fi
\ifx \binstitute  \undefined \def \binstitute#1{#1}\fi
\ifx \binstitutionaled  \undefined \def \binstitutionaled#1{#1}\fi
\ifx \bctitle  \undefined \def \bctitle#1{#1}\fi
\ifx \beditor  \undefined \def \beditor#1{#1}\fi
\ifx \bpublisher  \undefined \def \bpublisher#1{#1}\fi
\ifx \bbtitle  \undefined \def \bbtitle#1{#1}\fi
\ifx \bedition  \undefined \def \bedition#1{#1}\fi
\ifx \bseriesno  \undefined \def \bseriesno#1{#1}\fi
\ifx \blocation  \undefined \def \blocation#1{#1}\fi
\ifx \bsertitle  \undefined \def \bsertitle#1{#1}\fi
\ifx \bsnm \undefined \def \bsnm#1{#1}\fi
\ifx \bsuffix \undefined \def \bsuffix#1{#1}\fi
\ifx \bparticle \undefined \def \bparticle#1{#1}\fi
\ifx \barticle \undefined \def \barticle#1{#1}\fi
\bibcommenthead
\ifx \bconfdate \undefined \def \bconfdate #1{#1}\fi
\ifx \botherref \undefined \def \botherref #1{#1}\fi
\ifx \url \undefined \def \url#1{\textsf{#1}}\fi
\ifx \bchapter \undefined \def \bchapter#1{#1}\fi
\ifx \bbook \undefined \def \bbook#1{#1}\fi
\ifx \bcomment \undefined \def \bcomment#1{#1}\fi
\ifx \oauthor \undefined \def \oauthor#1{#1}\fi
\ifx \citeauthoryear \undefined \def \citeauthoryear#1{#1}\fi
\ifx \endbibitem  \undefined \def \endbibitem {}\fi
\ifx \bconflocation  \undefined \def \bconflocation#1{#1}\fi
\ifx \arxivurl  \undefined \def \arxivurl#1{\textsf{#1}}\fi
\csname PreBibitemsHook\endcsname

\bibitem[\protect\citeauthoryear{Kato}{1987}]{On-nonlinear-Schrodinger-equations}
\begin{barticle}
\bauthor{\bsnm{Kato}, \binits{T.}}:
\batitle{On nonlinear {S}chr{\"o}dinger equations}.
\bjtitle{Ann. Inst. H. Poincare, Physique Theorique}
\bvolume{46},
\bfpage{113}--\blpage{129}
(\byear{1987})
\end{barticle}
\endbibitem

\bibitem[\protect\citeauthoryear{Kato}{1995}]{On-nonlinear-Schrodinger-equations-II-HS-solutions-and-unconditional-well-posedness}
\begin{barticle}
\bauthor{\bsnm{Kato}, \binits{T.}}:
\batitle{On nonlinear {S}chr{\"o}dinger equations, {II}. {$H^s$}-solutions and
  unconditional well-posedness}.
\bjtitle{Journal d’Analyse Math{\'e}matique}
\bvolume{67}(\bissue{1}),
\bfpage{281}--\blpage{306}
(\byear{1995})
\end{barticle}
\endbibitem

\bibitem[\protect\citeauthoryear{Ginibre and
  Velo}{1979}]{On-a-class-of-nonlinear-Schrodinger-equations-I-The-Cauchy-problem-general-case}
\begin{barticle}
\bauthor{\bsnm{Ginibre}, \binits{J.}},
\bauthor{\bsnm{Velo}, \binits{G.}}:
\batitle{On a class of nonlinear {S}chr{\"o}dinger equations.{I}. the {C}auchy
  problem, general case}.
\bjtitle{Journal of Functional Analysis}
\bvolume{32}(\bissue{1}),
\bfpage{1}--\blpage{32}
(\byear{1979})
\end{barticle}
\endbibitem

\bibitem[\protect\citeauthoryear{Ginibre and
  Velo}{1985}]{The-global-Cauchy-problem-for-the-nonlinear-Schrodinger-equation-revisited}
\begin{bchapter}
\bauthor{\bsnm{Ginibre}, \binits{J.}},
\bauthor{\bsnm{Velo}, \binits{G.}}:
\bctitle{The global {C}auchy problem for the nonlinear {S}chr{\"o}dinger
  equation revisited}.
In: \bbtitle{Annales de l'Institut Henri Poincar{\'e} C, Analyse Non
  Lin{\'e}aire},
vol. \bseriesno{2},
pp. \bfpage{309}--\blpage{327}
(\byear{1985}).
\bcomment{Elsevier}
\end{bchapter}
\endbibitem

\bibitem[\protect\citeauthoryear{Nakamura and
  Wada}{2016}]{Modified-Strichartz-estimates-with-an-application-to-the-critical-nonlinear-Schrodinger-equation}
\begin{barticle}
\bauthor{\bsnm{Nakamura}, \binits{M.}},
\bauthor{\bsnm{Wada}, \binits{T.}}:
\batitle{Modified {S}trichartz estimates with an application to the critical
  nonlinear {S}chr{\"o}dinger equation}.
\bjtitle{Nonlinear Analysis}
\bvolume{130},
\bfpage{138}--\blpage{156}
(\byear{2016})
\end{barticle}
\endbibitem

\bibitem[\protect\citeauthoryear{Nakamura and
  Wada}{2019}]{Strichartz-type-estimates-in-mixed-Besov-spaces-with-application-to-critical-nonlinear-Schrodinger-equations}
\begin{barticle}
\bauthor{\bsnm{Nakamura}, \binits{M.}},
\bauthor{\bsnm{Wada}, \binits{T.}}:
\batitle{Strichartz type estimates in mixed besov spaces with application to
  critical nonlinear {S}chr{\"o}dinger equations}.
\bjtitle{Journal of Differential Equations}
\bvolume{267}(\bissue{5}),
\bfpage{3162}--\blpage{3180}
(\byear{2019})
\end{barticle}
\endbibitem

\bibitem[\protect\citeauthoryear{Cazenave}{2003}]{Semilinear-Schrodinger-equations}
\begin{bbook}
\bauthor{\bsnm{Cazenave}, \binits{T.}}:
\bbtitle{Semilinear {Schr\"odinger} {E}quations}
vol. \bseriesno{10}.
\bpublisher{American Mathematical Soc.}, \blocation{???}
(\byear{2003})
\end{bbook}
\endbibitem

\bibitem[\protect\citeauthoryear{Kenig
  et~al.}{1993}]{Small-solutions-to-nonlinear-Schrodinger-equations}
\begin{bchapter}
\bauthor{\bsnm{Kenig}, \binits{C.E.}},
\bauthor{\bsnm{Ponce}, \binits{G.}},
\bauthor{\bsnm{Vega}, \binits{L.}}:
\bctitle{Small solutions to nonlinear {S}chr{\"o}dinger equations}.
In: \bbtitle{Annales de l'Institut Henri Poincar{\'e} C, Analyse Non
  Lin{\'e}aire},
vol. \bseriesno{10},
pp. \bfpage{255}--\blpage{288}
(\byear{1993}).
\bcomment{Elsevier}
\end{bchapter}
\endbibitem

\bibitem[\protect\citeauthoryear{Laskin}{2002}]{Fractional-Schrodinger-equation}
\begin{barticle}
\bauthor{\bsnm{Laskin}, \binits{N.}}:
\batitle{Fractional {S}chr{\"o}dinger equation}.
\bjtitle{Physical Review E}
\bvolume{66}(\bissue{5}),
\bfpage{056108}
(\byear{2002})
\end{barticle}
\endbibitem

\bibitem[\protect\citeauthoryear{Laskin}{2000a}]{Fractional-quantum-mechanics}
\begin{barticle}
\bauthor{\bsnm{Laskin}, \binits{N.}}:
\batitle{Fractional quantum mechanics}.
\bjtitle{Physical Review E}
\bvolume{62}(\bissue{3}),
\bfpage{3135}
(\byear{2000})
\end{barticle}
\endbibitem

\bibitem[\protect\citeauthoryear{Laskin}{2000b}]{Fractional-quantum-mechanics-and-Levy-path-integrals}
\begin{barticle}
\bauthor{\bsnm{Laskin}, \binits{N.}}:
\batitle{Fractional quantum mechanics and l{\'e}vy path integrals}.
\bjtitle{Physics Letters A}
\bvolume{268}(\bissue{4-6}),
\bfpage{298}--\blpage{305}
(\byear{2000})
\end{barticle}
\endbibitem

\bibitem[\protect\citeauthoryear{Laskin}{2000c}]{Fractals-and-quantum-mechanics}
\begin{barticle}
\bauthor{\bsnm{Laskin}, \binits{N.}}:
\batitle{Fractals and quantum mechanics}.
\bjtitle{Chaos: An Interdisciplinary Journal of Nonlinear Science}
\bvolume{10}(\bissue{4}),
\bfpage{780}--\blpage{790}
(\byear{2000})
\end{barticle}
\endbibitem

\bibitem[\protect\citeauthoryear{Hong and
  Sire}{2015}]{On-Fractional-Schrodinger-Equations-in-sobolev-spaces}
\begin{barticle}
\bauthor{\bsnm{Hong}, \binits{Y.}},
\bauthor{\bsnm{Sire}, \binits{Y.}}:
\batitle{On fractional {S}chr{\"o}dinger equations in {S}obolev spaces}.
\bjtitle{Communications on Pure and Applied Analysis}
\bvolume{14}(\bissue{6}),
\bfpage{2265}--\blpage{2282}
(\byear{2015})
\end{barticle}
\endbibitem

\bibitem[\protect\citeauthoryear{Guo
  et~al.}{2008}]{Existence-of-the-global-smooth-solution-to-the-period-boundary-value-problem-of-fractional-nonlinear-Schrodinger-equation}
\begin{barticle}
\bauthor{\bsnm{Guo}, \binits{B.}},
\bauthor{\bsnm{Han}, \binits{Y.}},
\bauthor{\bsnm{Xin}, \binits{J.}}:
\batitle{Existence of the global smooth solution to the period boundary value
  problem of fractional nonlinear {Schr\"odinger} equation}.
\bjtitle{Applied Mathematics and Computation}
\bvolume{204}(\bissue{1}),
\bfpage{468}--\blpage{477}
(\byear{2008})
\end{barticle}
\endbibitem

\bibitem[\protect\citeauthoryear{Guo and
  Huo}{2010}]{Global-Well-Posedness-for-the-Fractional-Nonlinear-Schrodinger-Equation}
\begin{barticle}
\bauthor{\bsnm{Guo}, \binits{B.}},
\bauthor{\bsnm{Huo}, \binits{Z.}}:
\batitle{Global well-posedness for the fractional nonlinear {Schr\"odinger}
  equation}.
\bjtitle{Communications in Partial Differential Equations}
\bvolume{36}(\bissue{2}),
\bfpage{247}--\blpage{255}
(\byear{2010})
\end{barticle}
\endbibitem

\bibitem[\protect\citeauthoryear{Naber}{2004}]{Time-fractional-Schrodinger-equation}
\begin{barticle}
\bauthor{\bsnm{Naber}, \binits{M.}}:
\batitle{Time fractional {S}chr{\"o}dinger equation}.
\bjtitle{Journal of mathematical physics}
\bvolume{45}(\bissue{8}),
\bfpage{3339}--\blpage{3352}
(\byear{2004})
\end{barticle}
\endbibitem

\bibitem[\protect\citeauthoryear{Achar
  et~al.}{2013}]{Time-Fractional-Schrodinger-Equation-Revisited}
\begin{botherref}
\oauthor{\bsnm{Achar}, \binits{B.N.}},
\oauthor{\bsnm{Yale}, \binits{B.T.}},
\oauthor{\bsnm{Hanneken}, \binits{J.W.}}, et al.:
Time fractional {Schr\"odinger} equation revisited.
Advances in Mathematical Physics
\textbf{2013}
(2013)
\end{botherref}
\endbibitem

\bibitem[\protect\citeauthoryear{Caputo}{1967}]{Linear-Models-of-Dissipation-whose-Q-is-almost-Frequency-Independent-II}
\begin{barticle}
\bauthor{\bsnm{Caputo}, \binits{M.}}:
\batitle{Linear models of dissipation whose {Q} is almost frequency
  independent—{II}}.
\bjtitle{Geophysical journal international}
\bvolume{13}(\bissue{5}),
\bfpage{529}--\blpage{539}
(\byear{1967})
\end{barticle}
\endbibitem

\bibitem[\protect\citeauthoryear{Podlubny}{1998}]{Fractional-differential-equations-an-introduction-to-fractional-derivatives-fractional-differential-equations-to-methods-of-their-solution-and-some-of-their-applications}
\begin{bbook}
\bauthor{\bsnm{Podlubny}, \binits{I.}}:
\bbtitle{Fractional {D}ifferential {E}quations: {A}n {I}ntroduction to
  {F}ractional {D}erivatives, {F}ractional {D}ifferential {E}quations, to
  {M}ethods of {T}heir {S}olution and {S}ome of {T}heir {A}pplications}.
\bpublisher{Elsevier}, \blocation{???}
(\byear{1998})
\end{bbook}
\endbibitem

\bibitem[\protect\citeauthoryear{Kilbas
  et~al.}{2006}]{Theory-and-Applications-of-Fractional-Differential-Equations}
\begin{bbook}
\bauthor{\bsnm{Kilbas}, \binits{A.A.}},
\bauthor{\bsnm{Srivastava}, \binits{H.M.}},
\bauthor{\bsnm{Trujillo}, \binits{J.J.}}:
\bbtitle{Theory and {A}pplications of {F}ractional {D}ifferential {E}quations}
vol. \bseriesno{204}.
\bpublisher{elsevier}, \blocation{???}
(\byear{2006})
\end{bbook}
\endbibitem

\bibitem[\protect\citeauthoryear{G{\'o}rka
  et~al.}{2017}]{The-Time-Fractional-Schrodinger-Equation-on-Hilbert-Space}
\begin{barticle}
\bauthor{\bsnm{G{\'o}rka}, \binits{P.}},
\bauthor{\bsnm{Prado}, \binits{H.}},
\bauthor{\bsnm{Trujillo}, \binits{J.}}:
\batitle{The time fractional {S}chr{\"o}dinger equation on {H}ilbert space}.
\bjtitle{Integral Equations and Operator Theory}
\bvolume{87},
\bfpage{1}--\blpage{14}
(\byear{2017})
\end{barticle}
\endbibitem

\bibitem[\protect\citeauthoryear{G{\'o}rka
  et~al.}{2020}]{The-asymptotic-behavior-of-the-time-fractional-Schrodinger-equation-on-Hilbert-space}
\begin{botherref}
\oauthor{\bsnm{G{\'o}rka}, \binits{P.}},
\oauthor{\bsnm{Prado}, \binits{H.}},
\oauthor{\bsnm{Pons}, \binits{D.J.}}:
The asymptotic behavior of the time fractional {S}chr{\"o}dinger equation on
  {H}ilbert space.
Journal of Mathematical Physics
\textbf{61}(3)
(2020)
\end{botherref}
\endbibitem

\bibitem[\protect\citeauthoryear{Wang
  et~al.}{2012}]{Fractional-Schrodinger-equations-with-potential-and-optimal-controls}
\begin{barticle}
\bauthor{\bsnm{Wang}, \binits{J.}},
\bauthor{\bsnm{Zhou}, \binits{Y.}},
\bauthor{\bsnm{Wei}, \binits{W.}}:
\batitle{Fractional {S}chr{\"o}dinger equations with potential and optimal
  controls}.
\bjtitle{Nonlinear Analysis: Real World Applications}
\bvolume{13}(\bissue{6}),
\bfpage{2755}--\blpage{2766}
(\byear{2012})
\end{barticle}
\endbibitem

\bibitem[\protect\citeauthoryear{Peng
  et~al.}{2019}]{The-well-posedness-for-fractional-nonlinear-Schrodinger-equations}
\begin{barticle}
\bauthor{\bsnm{Peng}, \binits{L.}},
\bauthor{\bsnm{Zhou}, \binits{Y.}},
\bauthor{\bsnm{Ahmad}, \binits{B.}}:
\batitle{The well-posedness for fractional nonlinear {S}chr{\"o}dinger
  equations}.
\bjtitle{Computers \& Mathematics with Applications}
\bvolume{77}(\bissue{7}),
\bfpage{1998}--\blpage{2005}
(\byear{2019})
\end{barticle}
\endbibitem

\bibitem[\protect\citeauthoryear{Zhou
  et~al.}{2018}]{Duhamels-formula-for-time-fractional-Schrodinger-equations}
\begin{barticle}
\bauthor{\bsnm{Zhou}, \binits{Y.}},
\bauthor{\bsnm{Peng}, \binits{L.}},
\bauthor{\bsnm{Huang}, \binits{Y.}}:
\batitle{Duhamel's formula for time-fractional {S}chr{\"o}dinger equations}.
\bjtitle{Mathematical Methods in the Applied Sciences}
\bvolume{41}(\bissue{17}),
\bfpage{8345}--\blpage{8349}
(\byear{2018})
\end{barticle}
\endbibitem

\bibitem[\protect\citeauthoryear{Hicdurmaz and
  Ashyralyev}{2017}]{On-the-Stability-of-Time-Fractional-Schrodinger-Differential-Equations}
\begin{barticle}
\bauthor{\bsnm{Hicdurmaz}, \binits{B.}},
\bauthor{\bsnm{Ashyralyev}, \binits{A.}}:
\batitle{On the stability of time-fractional {S}chr{\"o}dinger differential
  equations}.
\bjtitle{Numerical Functional Analysis and Optimization}
\bvolume{38}(\bissue{10}),
\bfpage{1215}--\blpage{1225}
(\byear{2017})
\end{barticle}
\endbibitem

\bibitem[\protect\citeauthoryear{Lee}{2020}]{Strichartz-estimates-for-space-time-fractional-Schrodinger-equations}
\begin{barticle}
\bauthor{\bsnm{Lee}, \binits{J.B.}}:
\batitle{Strichartz estimates for space-time fractional {S}chr{\"o}dinger
  equations}.
\bjtitle{Journal of Mathematical Analysis and Applications}
\bvolume{487}(\bissue{2}),
\bfpage{123999}
(\byear{2020})
\end{barticle}
\endbibitem

\bibitem[\protect\citeauthoryear{Grande}{2019}]{Space-Time-Fractional-Nonlinear-Schrodinger-Equation}
\begin{barticle}
\bauthor{\bsnm{Grande}, \binits{R.}}:
\batitle{Space-time fractional nonlinear {S}chr{\"o}dinger equation}.
\bjtitle{SIAM Journal on Mathematical Analysis}
\bvolume{51}(\bissue{5}),
\bfpage{4172}--\blpage{4212}
(\byear{2019})
\end{barticle}
\endbibitem

\bibitem[\protect\citeauthoryear{Dong and
  Xu}{2008}]{Space-time-fractional-Schrodinger-equation-with-time-independent-potentials}
\begin{barticle}
\bauthor{\bsnm{Dong}, \binits{J.}},
\bauthor{\bsnm{Xu}, \binits{M.}}:
\batitle{Space--time fractional {Schr\"odinger} equation with time-independent
  potentials}.
\bjtitle{Journal of Mathematical Analysis and Applications}
\bvolume{344}(\bissue{2}),
\bfpage{1005}--\blpage{1017}
(\byear{2008})
\end{barticle}
\endbibitem

\bibitem[\protect\citeauthoryear{Liu and
  Jiang}{2018}]{A-numerical-method-for-solving-the-time-fractional-Schrodinger-equation}
\begin{barticle}
\bauthor{\bsnm{Liu}, \binits{N.}},
\bauthor{\bsnm{Jiang}, \binits{W.}}:
\batitle{A numerical method for solving the time fractional {S}chr{\"o}dinger
  equation}.
\bjtitle{Advances in Computational Mathematics}
\bvolume{44},
\bfpage{1235}--\blpage{1248}
(\byear{2018})
\end{barticle}
\endbibitem

\bibitem[\protect\citeauthoryear{Abu~Arqub}{2019}]{Application-of-Residual-Power-Series-Method-for-the-Solution-of-Time-fractional-Schrodinger-Equations-in-One-dimensional-Space}
\begin{barticle}
\bauthor{\bsnm{Abu~Arqub}, \binits{O.}}:
\batitle{Application of residual power series method for the solution of
  time-fractional {Schr\"odinger} equations in one-dimensional space}.
\bjtitle{Fundamenta Informaticae}
\bvolume{166}(\bissue{2}),
\bfpage{87}--\blpage{110}
(\byear{2019})
\end{barticle}
\endbibitem

\bibitem[\protect\citeauthoryear{Bhrawy
  et~al.}{2016}]{Jacobi-spectral-collocation-approximation-for-multi-dimensional-time-fractional-Schrodinger-equations}
\begin{barticle}
\bauthor{\bsnm{Bhrawy}, \binits{A.H.}},
\bauthor{\bsnm{Alzaidy}, \binits{J.F.}},
\bauthor{\bsnm{Abdelkawy}, \binits{M.A.}},
\bauthor{\bsnm{Biswas}, \binits{A.}}:
\batitle{Jacobi spectral collocation approximation for multi-dimensional
  time-fractional {Schr\"odinger} equations}.
\bjtitle{Nonlinear Dynamics}
\bvolume{84},
\bfpage{1553}--\blpage{1567}
(\byear{2016})
\end{barticle}
\endbibitem

\bibitem[\protect\citeauthoryear{Su
  et~al.}{2021}]{Dispersive-estimates-for-time-and-space-fractional-Schrodinger-equations}
\begin{barticle}
\bauthor{\bsnm{Su}, \binits{X.}},
\bauthor{\bsnm{Zhao}, \binits{S.}},
\bauthor{\bsnm{Li}, \binits{M.}}:
\batitle{Dispersive estimates for time and space fractional {S}chr{\"o}dinger
  equations}.
\bjtitle{Mathematical Methods in the Applied Sciences}
\bvolume{44}(\bissue{10}),
\bfpage{7933}--\blpage{7942}
(\byear{2021})
\end{barticle}
\endbibitem

\bibitem[\protect\citeauthoryear{Su
  et~al.}{2019}]{Local-well-posedness-of-semilinear-space-time-fractional-Schrodinger-equation}
\begin{barticle}
\bauthor{\bsnm{Su}, \binits{X.}},
\bauthor{\bsnm{Zhao}, \binits{S.}},
\bauthor{\bsnm{Li}, \binits{M.}}:
\batitle{Local well-posedness of semilinear space-time fractional
  {S}chr{\"o}dinger equation}.
\bjtitle{Journal of Mathematical Analysis and Applications}
\bvolume{479}(\bissue{1}),
\bfpage{1244}--\blpage{1265}
(\byear{2019})
\end{barticle}
\endbibitem

\bibitem[\protect\citeauthoryear{Wang
  et~al.}{2022}]{Well-posedness-and-blow-up-results-for-a-class-of-nonlinear-fractional-Rayleigh-Stokes-problem}
\begin{barticle}
\bauthor{\bsnm{Wang}, \binits{J.N.}},
\bauthor{\bsnm{Alsaedi}, \binits{A.}},
\bauthor{\bsnm{Ahmad}, \binits{B.}},
\bauthor{\bsnm{Zhou}, \binits{Y.}}:
\batitle{Well-posedness and blow-up results for a class of nonlinear fractional
  {R}ayleigh-{S}tokes problem}.
\bjtitle{Advances in Nonlinear Analysis}
\bvolume{11}(\bissue{1}),
\bfpage{1579}--\blpage{1597}
(\byear{2022})
\end{barticle}
\endbibitem

\bibitem[\protect\citeauthoryear{de~Andrade
  et~al.}{2020}]{Well-posedness-results-for-a-class-of-semilinear-time-fractional-diffusion-equations}
\begin{barticle}
\bauthor{\bsnm{Andrade}, \binits{B.}},
\bauthor{\bsnm{Van~Au}, \binits{V.}},
\bauthor{\bsnm{O’Regan}, \binits{D.}},
\bauthor{\bsnm{Tuan}, \binits{N.H.}}:
\batitle{Well-posedness results for a class of semilinear time-fractional
  diffusion equations}.
\bjtitle{Zeitschrift f{\"u}r angewandte Mathematik und Physik}
\bvolume{71},
\bfpage{1}--\blpage{24}
(\byear{2020})
\end{barticle}
\endbibitem

\bibitem[\protect\citeauthoryear{He
  et~al.}{2021}]{On-well-posedness-of-semilinear-Rayleigh-Stokes-problem-with-fractional-derivative-on}
\begin{barticle}
\bauthor{\bsnm{He}, \binits{J.W.}},
\bauthor{\bsnm{Zhou}, \binits{Y.}},
\bauthor{\bsnm{Peng}, \binits{L.}},
\bauthor{\bsnm{Ahmad}, \binits{B.}}:
\batitle{On well-posedness of semilinear {R}ayleigh-{S}tokes problem with
  fractional derivative on {$\mathbb{R}^N$}}.
\bjtitle{Advances in Nonlinear Analysis}
\bvolume{11}(\bissue{1}),
\bfpage{580}--\blpage{597}
(\byear{2021})
\end{barticle}
\endbibitem

\bibitem[\protect\citeauthoryear{Grafakos
  et~al.}{2009}]{Modern-Fourier-Analysis}
\begin{bbook}
\bauthor{\bsnm{Grafakos}, \binits{L.}}, \betal:
\bbtitle{Modern {F}ourier {A}nalysis}
vol. \bseriesno{250}.
\bpublisher{Springer}, \blocation{???}
(\byear{2009})
\end{bbook}
\endbibitem

\bibitem[\protect\citeauthoryear{Gulisashvili and
  Kon}{1996}]{Exact-smoothing-properties-of-Schrodinger-semigroups}
\begin{barticle}
\bauthor{\bsnm{Gulisashvili}, \binits{A.}},
\bauthor{\bsnm{Kon}, \binits{M.A.}}:
\batitle{Exact smoothing properties of {S}chr{\"o}dinger semigroups}.
\bjtitle{American Journal of Mathematics}
\bvolume{118}(\bissue{6}),
\bfpage{1215}--\blpage{1248}
(\byear{1996})
\end{barticle}
\endbibitem

\bibitem[\protect\citeauthoryear{Henry}{2006}]{Geometric-Theory-of-Semilinear-Parabolic-Equations}
\begin{bbook}
\bauthor{\bsnm{Henry}, \binits{D.}}:
\bbtitle{Geometric {T}heory of {S}emilinear {P}arabolic {E}quations}
vol. \bseriesno{840}.
\bpublisher{Springer}, \blocation{???}
(\byear{2006})
\end{bbook}
\endbibitem

\bibitem[\protect\citeauthoryear{Mainardi}{2022}]{Fractional-Calculus-and-Waves-in-Linear-Viscoelasticity}
\begin{bbook}
\bauthor{\bsnm{Mainardi}, \binits{F.}}:
\bbtitle{Fractional {C}alculus and {W}aves in {L}inear {V}iscoelasticity: {A}n
  {I}ntroduction to {M}athematical {M}odels}.
\bpublisher{World Scientific}, \blocation{???}
(\byear{2022})
\end{bbook}
\endbibitem

\bibitem[\protect\citeauthoryear{Grafakos
  et~al.}{2008}]{Classical-fourier-analysis}
\begin{bbook}
\bauthor{\bsnm{Grafakos}, \binits{L.}}, \betal:
\bbtitle{Classical {F}ourier {A}nalysis}
vol. \bseriesno{2}.
\bpublisher{Springer}, \blocation{???}
(\byear{2008})
\end{bbook}
\endbibitem

\bibitem[\protect\citeauthoryear{Bahouri}{2011}]{Fourier-Analysis-and-Nonlinear-Partial-Differential-Equations}
\begin{bbook}
\bauthor{\bsnm{Bahouri}, \binits{H.}}:
\bbtitle{Fourier {A}nalysis and {N}onlinear {P}artial {D}ifferential
  {E}quations}.
\bpublisher{Springer}, \blocation{???}
(\byear{2011})
\end{bbook}
\endbibitem

\end{thebibliography}
\end{document}